\documentclass[graybox]{svmult}

\smartqed
\usepackage{mathptmx}       
\usepackage{helvet}         
\usepackage{courier}        
\usepackage{type1cm}        
\usepackage{makeidx}         
\usepackage{graphicx}        
\usepackage{multicol}        
\usepackage[bottom]{footmisc}

\usepackage{amsmath}

\usepackage[utf8]{inputenc}
\usepackage{amssymb}
\usepackage{epstopdf}
\usepackage{array}
\usepackage[english]{babel}
\usepackage[T1]{fontenc}

\graphicspath{{./}}

\spnewtheorem{assumption}{Assumption}{\bf}{\it}

\makeindex             

\newcommand{\R}{\mathbb{R}}
\newcommand\tilc{{\widetilde{c}}}
\newcommand\tilg{{\widetilde{\gamma}}}

\begin{document}

\title*{A Numerical Method to solve Optimal Transport Problems with Coulomb Cost} 
\author{Jean-David  Benamou \and Guillaume  Carlier \and Luca Nenna}
\institute{Jean-David Benamou, Luca Nenna \at INRIA, MOKAPLAN, Domaine de Voluceau Le Chesnay, FRANCE, \email{jean-david.benamou/luca.nenna@inria.fr}
\and Guillaume Carlier \at CEREMADE, Universit\'e Paris Dauphine \email{carlier@ceremade.dauphine.fr}}
%
%

\maketitle

\abstract{
In this paper, we present a numerical method, based on iterative Bregman projections, to solve the optimal transport problem with Coulomb cost.
This  is related to the strong interaction limit of Density Functional Theory.
The first idea is to introduce an entropic regularization of the Kantorovich formulation of the Optimal Transport problem.
The regularized problem then corresponds to the projection of a vector on the intersection of the constraints with respect to the Kullback-Leibler distance.
Iterative Bregman projections on each marginal constraint are explicit which enables us to approximate the optimal transport plan. We validate the numerical method against   analytical test cases.}


\section{Introduction}

\subsection{On Density functional theory}

Quantum mechanics for a molecule with $N$ electrons boils down to the many-electron Schr\"{o}dinger equation for a  wave function $\psi\in\ L^{2}(\mathbb{R}^{3N};\mathbb{C})$ (in this paper, we neglect the spin variable).
The  limit  of this approach is computational~: in order to predict the chemical behaviour of $H_{2}O$ ($10$ electrons) using a $10$ gridpoints discretization of $\mathbb{R}$, we need to  solve
the Schr\"{o}dinger equation on $10^{30}$ gridpoints. This is why Hohenberg, Kohn and Sham introduced, in \cite{HK} and \cite{KS}, the Density Functional Theory (DFT) as an approximate computational method for solving the Schr\"{o}dinger equation  at a more reasonable cost.

The main idea of the DFT is to compute only the marginal density for one electron
\[\rho(x_{1})=\int \gamma_{N}(x_1, x_2 \cdots, x_N)  dx_{2}\cdots dx_{N},\]
 where $\gamma_{N}=\lvert\psi(x_{1},\cdots,x_{N})\rvert^{2}$ is the joint probability density of electrons at positions
$x_{1},\cdots,x_{N}\in\mathbb{R}^{3}$, instead of the full wave function $\psi$.
One  scenario of interest for the DFT is when the repulsion between the electrons largely dominates over the kinetic energy. In this case, the problem can, at least formally, be reformulated as an Optimal Transport  (OT) problem as emphasized in the pioneering works of Buttazzo, De Pascale and Gori-Giorgi \cite{ButDeGo} and Cotar, Friesecke and Kl\"uppelberg \cite{CoFrieKl}.


\subsection{Optimal Transport}
\label{OT}

Before discussing the link between  DFT and OT, let us recall the standard optimal transport problem  and its
extension to the multi-marginal framework.  Given two probability distributions $\mu$ and $\nu$ (on $\R^d$, say) and a transport cost $c$: $\R^d\times \R^d\to \R$, the optimal transport problem consists in finding the cheapest way to transport $\mu$ to $\nu$ for the cost $c$. A transport map between $\mu$ and $\nu$ is a Borel map $T$ such that $T_\#\mu=\nu$ i.e. $\nu(A)=\mu(T^{-1}(A))$ for every Borel subset $A$ of $\R^d$. The Monge problem (which dates back to 1781 when Monge \cite{Monge} posed the problem of finding the optimal way to move a pile of dirt to a hole of the same volume) then reads
\begin{equation}
\label{Mo}
 \min_{T_\# \mu=\nu} \int_{\R^d} c(x,T(x))\mu(dx).
\end{equation}
This is a delicate problem since the mass conservation constraint $T_\# \mu=\nu$ is highly nonlinear (and the feasible set may even be empty for instance if $\mu$ is a Dirac mass and $\nu$ is not). This is why, in 1942, Kantorovich \cite{Kanth42} proposed a relaxed formulation of (\ref{Mo}) which allows mass splitting
\begin{equation}
\label{K}
 \min_{\gamma\in\Pi(\mu,\nu)}\int_{\R^d\times\R^d} c(x,y)\gamma(dx,dy)
\end{equation}
where $\gamma\in\Pi(\mu,\nu)$ consists of all probability measures on $\R^d\times \R^d$ having $\mu$ and $\nu$ as marginals, that is:
\begin{alignat}{3}
 \gamma(A\times\mathbb{R^d})&=\mu(A),\quad\forall A \mbox{ Borel subset of } \R^d,\\
 \gamma(\mathbb{R^d}\times B)&=\nu(B),\quad\forall B   \mbox{ Borel subset of }\R^d.
\end{alignat}
Note that this is a linear programming problem and that there exists solutions under very mild assumptions (e.g. $c$ continuous and $\mu$ and $\nu$ compactly supported).
A minimizing $\gamma$ in (\ref{K}) is called an optimal transport plan and it gives the probability that a mass element in $x$ be transported in $y$. Let us remark that if $T$ is a transport map then it induces a transport plan $\gamma_{T}(x,y):=\mu(x)\delta(y-T(x))$ so if an optimal plan of (\ref{K}) has the form $\gamma_{T}$ (which means that no splitting of mass occurs and $\gamma$ is concentrated on the graph of $T$) then $T$ is actually an optimal transport map i.e. a solution to (\ref{Mo}).
The linear problem (\ref{K})  also has a convenient dual formulation
\begin{equation}
 \max_{u,v\lvert u(x)+v(y)\leq c(x,y)}\int_{\R^d}u(x)\mu(dx)+\int_{\mathbb{R}^d}v(y)\nu(dy)
\end{equation}
where $u(x)$ and $v(y)$ are the so-called Kantorovich potentials. OT theory for two marginals has developed very rapidly in the 25 last years, there are well known conditions on $c$, $\mu$ and $\nu$ which guarantee that there is a unique optimal plan which is in fact induced by a map (e.g. $c=\vert x-y\vert^2$ and $\mu$ absolutely continuous, see Brenier \cite{bre}) and we refer to the textbooks of Villani \cite{Villani03, Villani09} for a detailed exposition.

\smallskip

Let us now consider the so-called multi-marginal problems i.e. OT problems involving $N$ marginals $\mu_1, \cdots, \mu_N$ and a cost $c$ : $\R^{dN}\to \R$, which leads to the following generalization of (\ref{K}) 
\begin{equation}
 \label{MMKP}
 \min_{\gamma\in\Pi(\mu_{1},\cdots,\mu_{N})}\int_{\mathbb{R^{d\times N}}} c(x_{1},\cdots,x_{N})\gamma(dx_{1},\cdots,dx_{N})
\end{equation}
where $\Pi(\mu_{1},\cdots,\mu_{N})$ is the set of probability measures on $(\R^d)^N$ having $\mu_1, \cdots, \mu_N$ as marginals. The corresponding Monge problem then becomes
\begin{equation}
 \label{MMP}
 \min_{{T_i}_\#\mu_1=\mu_i, \; i=2, \cdots, N}  \int_{\R^d} c(x_1,T_{2}(x_1),\cdots,T_{N}(x_1))\mu_{1}(dx_1).
\end{equation}
Such multi-marginals problems first appeared in the work of Gangbo and  {\'S}wi{\c{e}}ch \cite{gansw} who solved the quadratic cost  case and proved the existence of Monge solutions. In recent years, there has been a lot of interest in such multi-marginal problems because they arise naturally in many different settings such as economics \cite{carlierekeland}, \cite{Pass-match}, polar factorization of vector fields and theory of monotone maps \cite{ghou-mau} and, of course, DFT \cite{ButDeGo, CoFrieKl, CPM,friesecke2013n,mendl2013kantorovich,cotar2013infinite},  as is recalled below. Few results are known about the structure of optimal plans  for (\ref{MMP}) apart from the general results of Brendan Pass \cite{pass}, in particular the case of \emph{repulsive costs} such as the Coulomb's cost from DFT is an open problem.


\smallskip

The paper is structured as follows.
In Section \ref{DFTandOT}, we recall the link between  Density Functional Theory and Optimal Transportation and we present some analytical solutions of the OT problem (e.g. optimal maps for radially symmetric marginals, for 2 electrons). In Section \ref{AlProj}, we introduce a numerical method, based on iterative Bregman projections, and an algorithm which aims at refining the mesh where the transport plan is concentrated. In section \ref{NumRes} we present some numerical results.  Section \ref{ccl} concludes.

\section{From Density Functional Theory to Optimal Transportation}
\label{DFTandOT}

\subsection{Optimal Transportation with Coulomb cost }


In Density Functional Theory \cite{HK} the ground state energy of a system (with $N$ electrons) is obtained by minimizing the following functional w.r.t. the electron density $\rho(x)$:
\begin{equation}
 \label{eq4}
 E[\rho]=\min_{\rho\in\mathcal{R}}F_{HK}[\rho]+\int v_{ext}(x)\rho(x)dr
\end{equation}
where $\mathcal{R}=\{ \rho:\mathbb{R}^{3}\rightarrow\mathbb{R} \lvert \rho\geq 0,\sqrt{\rho}\in H^{1}(\mathbb{R}^{3}),\int_{\mathbb{R}^{3}}\rho(x)dx=N\}$,

$v_{ext}:=-\dfrac{Z}{\lvert x-R\rvert}$ is the electron-nuclei potential ($Z$ and $R$ are the charge and the position of the nucleus, respectively)  and $F_{HK}$ is the so-called Hohenberg-Kohn which is defined by minimizing over all wave functions $\psi$ which yield $\rho$: 

\begin{equation}
 \label{eq5}
 F_{HK}[\rho]=\min_{\psi \rightarrow \rho} \hbar^{2}T[\psi]+V_{ee}[\psi]
\end{equation}

where 
$\hbar^{2}$ is a semiclassical constant factor,

\begin{center}
$T[\psi]=\dfrac{1}{2}\int\cdots\int\sum_{i=1}^{N}\lvert \nabla_{x_{i}} \psi\rvert^{2}dx_{1}\cdots dx_{N}$ 
\end{center}
is the kinetic energy   and 
\begin{center}
$V_{ee}= \int\cdots\int\ \sum_{i=1}^{N}\sum_{j>i}^{N}\dfrac{1}{\lvert x_{i}-x_{j}\rvert} \lvert \psi \rvert^{2}  dx_{1}\cdots dx_{N}$ 
\end{center}
is the Coulomb repulsive energy operator.

Let us now consider the  \textit{Semiclassical} limit
\begin{center}
$\lim_{\hbar\rightarrow 0} \min_{\psi \rightarrow \rho} \hbar^{2}T[\psi]+V_{ee}[\psi] $
\end{center}
and  assume that taking the minimum over $\psi$ commutes with passing to the limit $\hbar\rightarrow 0$ (Cotar, Friesecke and Kl\"{u}ppelberg in \cite{CoFrieKl} proved it for $N=2$),
we obtain the following functional
\begin{equation}
 \label{sce}
V_{ee}^{SCE}[\rho]=\min_{\psi\rightarrow \rho}\int\cdots\int\ \sum_{i=1}^{N}\sum_{j>i}^{N}\dfrac{1}{\lvert x_{i}-x_{j}\rvert} \lvert \psi \rvert^{2}  dx_{1}\cdots dx_{N}
\end{equation}
where $V_{ee}^{SCE}$ is the minimal Coulomb repulsive energy whose minimizer characterizes the state of \textit{Strictly Correlated Electrons}(SCE). 

Problem (\ref{sce})  gives rise to a multi-marginal optimal transport problem as (\ref{MMKP}) by considering that

 \begin{itemize}
 \item according to the indistinguishability of electrons, all the marginals are equal to $\rho$,

 \item  the cost function is given the electron-electron Coulomb repulsion,
  \begin{equation}
 \label{eq7}
  c(x_{1},...,x_{N})=\sum_{i=1}^{N}\sum_{j>i}^{N}\dfrac{1}{\lvert x_{i}-x_{j}\rvert},
 \end{equation}
 
 \item we refer to $\gamma_{N}=\lvert\psi(x_{1},\cdots,x_{N})\rvert^{2}$  (which is the joint probability density of electrons at positions $x_{1},\cdots,x_{N}\in\mathbb{R}^{3}$) as the transport plan. 
\end{itemize}
The Coulomb cost function (\ref{eq7}) is different from the  costs usually considered in OT as it is not bounded at the origin and it decreases with distance. So it requires a generalized
formal framework, but this is beyond the purpose of this work (see \cite{ButDeGo} and \cite{CoFrieKl}).
Finally  (\ref{sce}) can be re-formulated as a Kantorovich problem 
\begin{equation}
\label{kant}
V_{ee}^{SCE}[\rho]=\min_{\pi_{i}(\gamma_{N})=\rho,i=1,\cdots,N}\int_{\mathbb{R}^{3N}}c(x_{1},\cdots,x_{N})\gamma_{N}(x_{1},\cdots,x_{N})dx_{1}\cdots dx_{N}
\end{equation}
where
\begin{center}
$\pi_{i}(\gamma_{N})=\int_{\mathbb{R}^{3(N-1)}}\gamma_{N}(x_{1},\cdots,x_{i},\cdots,x_{N})dx_{1},\cdots,dx_{i-1},dx_{i+1},\cdots,dx_{N}$ 
\end{center}
is the $i-$th marginal.
As mentioned in section \ref{OT} if the optimal transport plan $\gamma_{N}$ has the following form 

\begin{equation}
 \gamma_{N}(x_{1},\cdots,x_{N})=\rho(x_{1}) \delta(x_{2}-f_{2}^\star(x_{1}))\cdots\delta(x_{N}-f_{N}^\star(x_{1}))
\end{equation}
then the functions $f_{i}^\star:\mathbb{R}^{3}\rightarrow\mathbb{R}^{3}$ are the optimal transport maps (or \textit{co-motion} functions) of the Monge problem 
\begin{equation}
\begin{split}
 \label{eq6}
 & V_{ee}^{SCE}[\rho]=\min_{\{f_{i}:\mathbb{R}^{3}\rightarrow\mathbb{R}^{3}\}_{i=1}^{N}}\int\sum_{i=1}^{N}\sum_{j>i}^{N}\dfrac{1}{\lvert f_{i}(x)-f_{j}(x)\rvert}\rho(x) dx \\ & s.t.\quad {f_i}_\#\rho=\rho, \; i=2, ...,N, \quad f_{1}(x)=x.
\end{split}
 \end{equation}
 
\begin{remark}{(Physical meaning of the co-motion function)}
$f_{i}(x)$  determine the position of the $i$-th electron in terms of $x$ which
is the position of the \textquotedblleft 1st\textquotedblright electron~: $V_{ee}^{SCE}$ defines a system with the maximum possible correlation between the relative electronic positions.
 
 \end{remark}
 
In full generality,  problem (\ref{eq6}) is  delicate and  proving the existence of the co-motion functions is difficult. 
However,  the co-motion functions can be obtained via semianalytic formulations for spherically symmetric atoms and strictly 1D systems (see \cite{CoFrieKl}, \cite{SeidlGoriSavin}, \cite{MaletGori}, \cite{CPM})
and we will give some examples in the following section.

Problem (\ref{kant}) admits a useful dual formulation in which the so called Kantorovich potential $u$ plays a central role
\begin{equation}
\label{eq8}
V_{ee}^{SCE}=\max_{u}\{N\int u(x)\rho(x)dx\quad s.t.\quad \sum_{i=1}^{N}u(x_{i})\leq c(x_{1},...,x_{N})\}.
\end{equation}
Because $c$ is invariant by permutation, there is a single dual Kantorovich potential for all all marginal constraints. 
Moreover, this potential $u(x)$ is related to the co-motion functions via the classical equilibrium equation (see \cite{SeidlGoriSavin})
\begin{equation}
\label{equi}
\nabla u(x)=-\sum_{i=2}^{N}\dfrac{x-f_{i}(x)}{\lvert x -f_{i}(x)\rvert^{3}}.
\end{equation}
\begin{remark}{(Physical meaning of (\ref{equi}))}
The gradient of the Kantorovich potential equals the total net force exerted on the electron in $x$ by electrons in $f_{2}(x),\cdots,f_{N}(x)$.
\end{remark}


\subsection{Analytical Examples}\label{anal}
\subsubsection{The case $N=2$ and $d=1$}
In order to better understand the problem we have formulated in the previous section, we recall some analytical examples (see \cite{ButDeGo} for the details).

Let us consider 2 particles in one dimension  and  marginals

\begin{equation}
 \label{eq9}
 \rho_{1}(x)=\rho_{2}(x)=\begin{cases} a\quad if\lvert x\rvert \leq a/2 \\ 0\quad otherwise.\end{cases}
\end{equation}

After a few computations, we obtain the following associated co-motion function

\begin{equation}
 f(x)=\begin{cases} x+\frac{a}{2} \\ x-\frac{a}{2} \end{cases}.
\end{equation}

If we take

\begin{equation}
 \label{eq10}
 \rho_{1}(x)=\rho_{2}(x)=\dfrac{a-\lvert x\rvert}{a^{2}}\quad defined\quad in \quad [-a,a],
\end{equation}

we get

\begin{equation}
 f(x)=\dfrac{x}{\lvert x\rvert}(\sqrt{2a\lvert x\rvert-x^{2}}-a)\quad on\quad [-a,a]
\end{equation}

Figure \ref{figure:fig4} shows the co-motion functions for (\ref{eq9}) and (\ref{eq10}).

\begin{figure}[h!]
\begin{tabular}{@{}c@{\hspace{1mm}}c@{}}
\centering
\includegraphics[ scale=0.35]{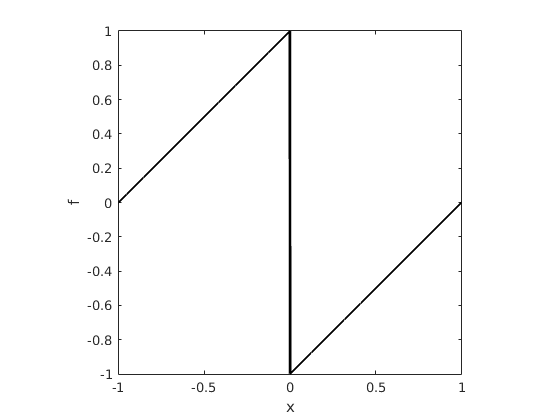} &
\includegraphics[ scale=0.35]{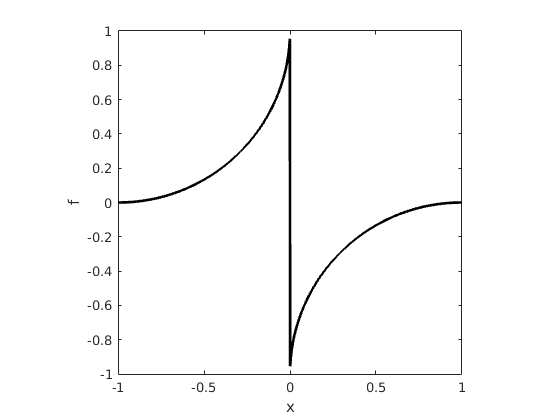}
\end{tabular}
\caption{\textit{Right: Co-motion function for (\ref{eq9}) with  $a=2$.  \textit{Left}: Co-motion function for (\ref{eq10}) with $a=1$.}}
\label{figure:fig4}
\end{figure}

\subsubsection{The case $N>2$ and $d=1$}
\label{multi1D}
In \cite{CPM}, the authors proved the existence of optimal transport maps for  problem (\ref{eq6}) in
dimension $d=1$ and provided an explicit construction of the optimal maps. Let $\rho$ be the normalized electron density and $-\infty=x_{0}<x_{1}<\cdots<x_{N}=+\infty$ be such that 
\begin{center}
 $\int_{x_{i}}^{x_{i+1}}\rho(x)dx=1/N$ $\forall i=0,\cdots,N-1$.
\end{center}
Thus, there exists a unique increasing function $\tilde{f}:\mathbb{R}\rightarrow\mathbb{R}$ on each interval $[x_{i},x_{i+1}]$ such that for every test-function $\varphi$ one has
\begin{alignat}{3}
\int_{[x_{i},x_{i+1}]}  \varphi( \tilde{f}(x))  \rho(x)dx&=\int_{[x_{i+1},x_{i+2}]} \varphi(x) \rho(x)dx\quad\forall i=0,\cdots,N-2, \\ 
\int_{[x_{N-1},x_{N}]} \varphi( \tilde{f}(x))\rho(x)   dx&=\int_{[x_{0},x_{1}]}   \varphi(x)  \rho(x)dx,
\end{alignat}
The optimal maps are then given by 
\begin{alignat}{3}
f_{2}(x)&=\tilde{f}(x)\\
f_{i}(x)&=f_{2}^{(i)}(x)\quad \forall i=2,\cdots,N,
\end{alignat}
where $f_{2}^{(i)}$ stands for the $i-$th composition of $f_2$ with itself. Here, we present an example given in \cite{ButDeGo}. We consider the case where $\rho$ is the Lebesgue measure on the unit interval $I=[0,1]$,  the construction above gives the following optimal co-motion functions

\begin{equation}
\label{eq11}
\begin{array}{l}
\mbox{ $f_{2}(x)=\begin{cases} x+1/3\quad if\quad x\leq 2/3 \\ x-2/3\quad if\quad x>2/3\end{cases}$, }\\
 \mbox{$f_{3}(x)=f_2(f_2(x))=\begin{cases} x+2/3\quad if\quad x\leq 1/3 \\ x-1/3\quad if\quad x>1/3\end{cases}$.}
 \end{array} 
\end{equation}

Furthermore, we know that the Kantorovich potential $u$ satisfies the relation (here we take $N=3$)

\begin{equation}
\label{eq12}
  u'(x)=-\sum_{i=2}^{N}\dfrac{x-f_{i}(x)}{\lvert x-f_{i}(x)\rvert^{3}}
\end{equation}
and by substituting the co-motion functions in (\ref{eq12}) (and integrating it) we get
\begin{equation}
 \label{13}
 u(x)=\begin{cases} \frac{45}{4}x\quad\quad & 0\leq x\leq1/3 \\ \frac{15}{4}\quad\quad & 1/3\leq x\leq 2/3 \\ -\frac{45}{4}x+\frac{45}{4}\quad\quad & 2/3\leq x\leq 1\end{cases}
\end{equation}

Figure  \ref{figure:fig6}  illustrates this example. \\

When $N\geq 4$ similar arguments as above can be developed and we can similarly compute the co-motion functions and the Kantorovich potential.

\begin{figure}[htbp]
\begin{tabular}{@{}c@{\hspace{1mm}}c@{\hspace{1mm}}c@{}}
\centering
\includegraphics[ scale=0.14]{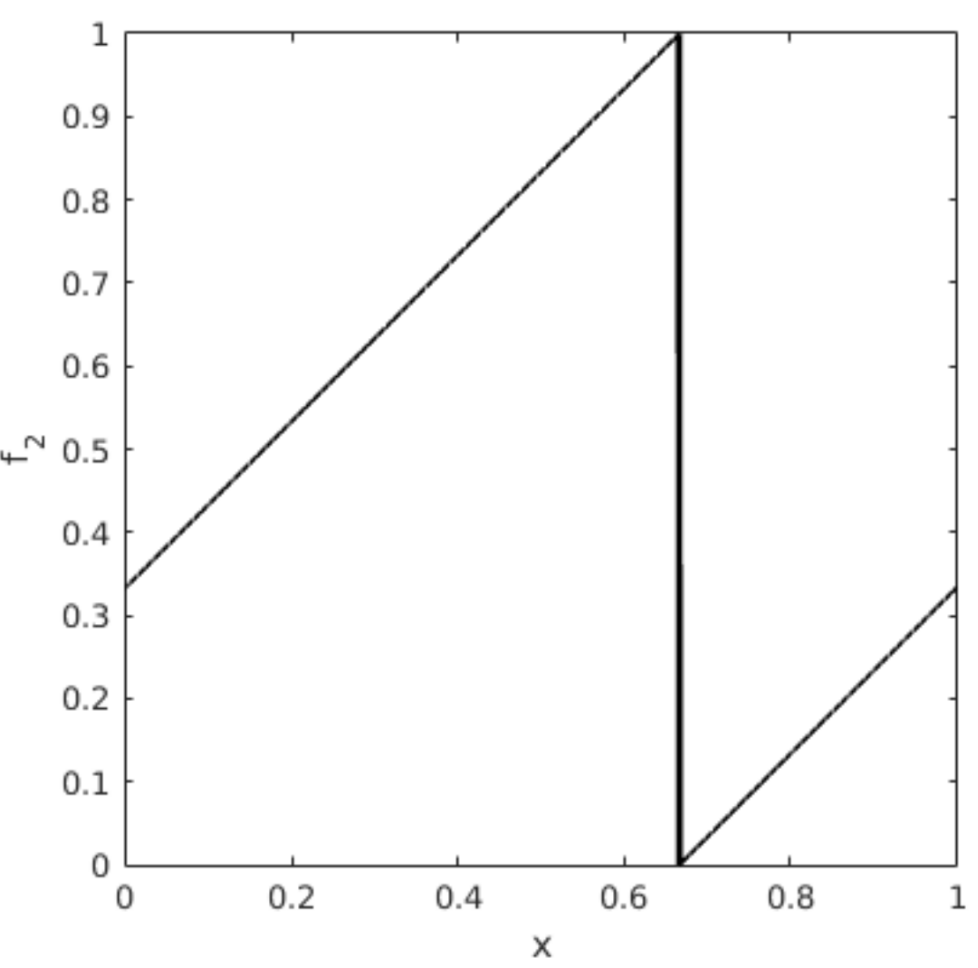} &
\includegraphics[ scale=0.14]{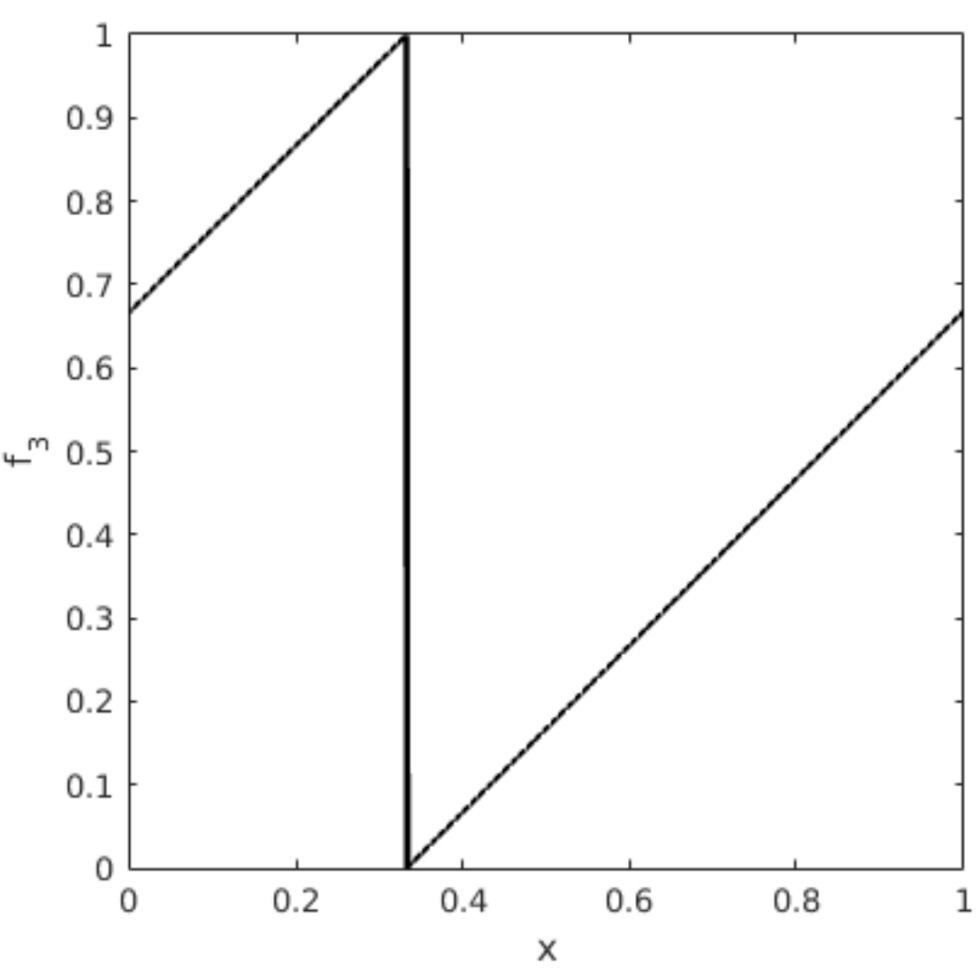}&
\includegraphics[ scale=0.123]{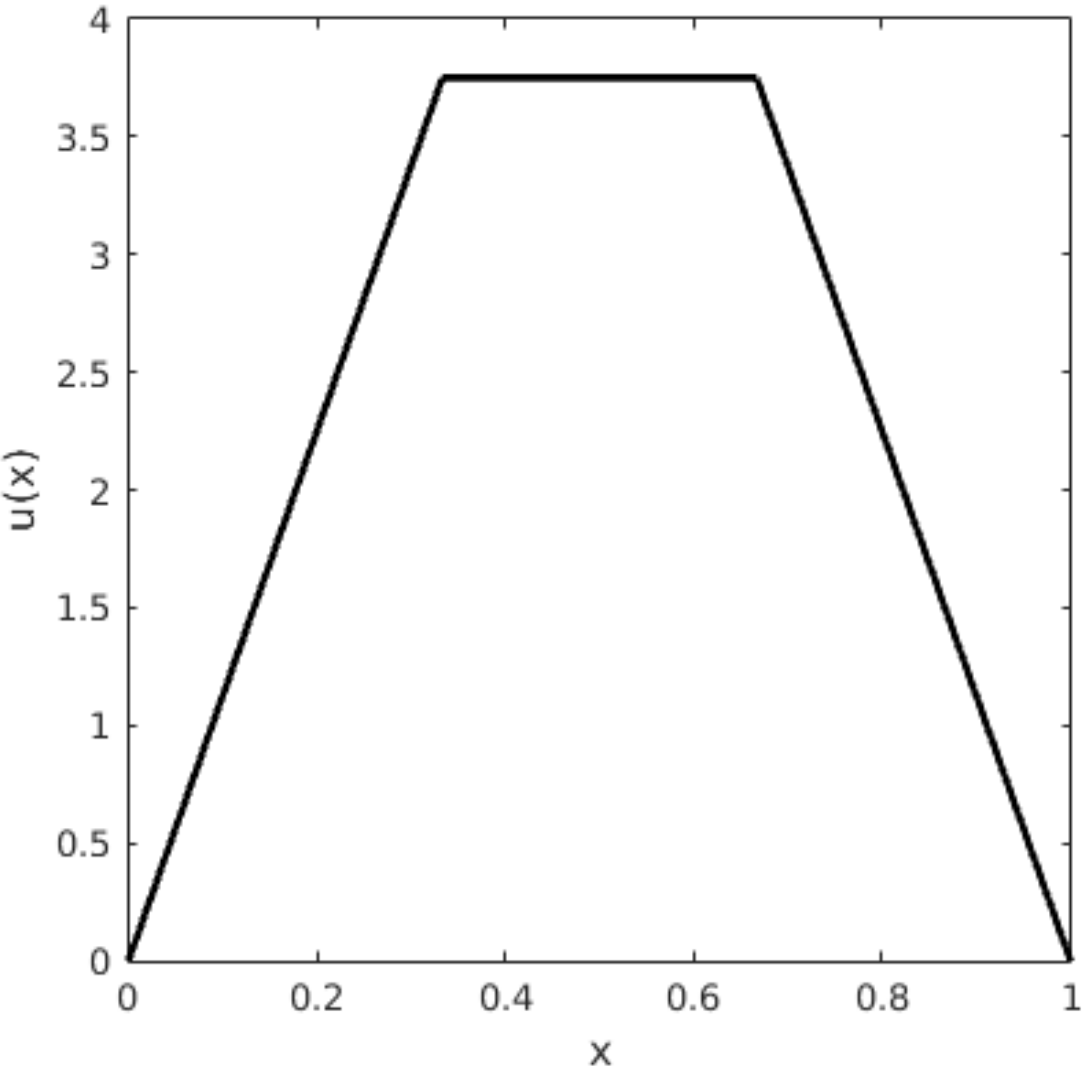} \\

\end{tabular}

\caption{\textit{Right: co-motion function $f_{2}$ for (\ref{eq11}). Center: co-motion function $f_{3}$ for (\ref{eq11}).
Left: Kantorovich Potential $u(x)$  (\ref{13}).}}
\label{figure:fig6}
\end{figure}


\subsubsection{The radially symmetric marginal case for $N=2$,  $d\geq 2$}
\label{radial}
We discuss now the radial $d-$dimensional ($d\geq 2$) case for $N=2$. We assume that the marginal $\rho$ is radially symmetric, then we recall the following theorem from \cite{CoFrieKl}:

\begin{theorem}{\cite{CoFrieKl}}
 Suppose that $\rho(x)=\rho(\lvert x\rvert)$, then the optimal transport map  is given by
 \begin{equation}
 f^{\star}(x)=\dfrac{x}{\lvert x\rvert}g(\vert x\vert),\quad x\in\R^d,
 \end{equation}
 with $g(r)=-F_{2}^{-1}(F_{1}(r))$, $F_{1}(t):=C(d)\int_{0}^{t}\rho(s)s^{d-1}ds$, $F_{2}(t):=C(d)\int_{t}^{\infty}\rho(s)s^{d-1}ds$ where $C(d)$ denotes the measure of $S^{d-1}$, the unit sphere in $\R^d$.

\end{theorem}

\begin{example}{(Spherical coordinates system)}
If $\rho$ is radially symmetric $\rho(x)=\rho(\vert x\vert)$, it is convenient to work in spherical coordinates and then to set for every radius $r>0$
\begin{equation}\label{deflambda}
\lambda(r)=C(d) r^{d-1} \rho(r)
\end{equation}
so that for every test-function $\varphi$ we have
\[\int_{\R^d} \varphi(x) \rho(\vert x \vert) dx=\int_0^{+\infty} \Big(\int_{S^{d-1}} \varphi(r ,\omega) \frac{ d \sigma(\omega)}{C_d} \Big) \lambda(r) dr\]
with $C(d)$ the measure of $S^{d-1}$ and $\sigma$ the $d-1$ measure on $S^{d-1}$ which in particular implies that $\lambda:=\vert. \vert_\# \rho$ i.e.
\begin{equation}\label{deflambda2}
\int_{\R^d} \varphi(\vert x\vert ) \rho(\vert x \vert) dx=\int_0^{+\infty} \varphi(r) \lambda(r) dr, \; \forall \varphi\in C_c(\R_+).
\end{equation}
The radial part of the optimal co-motion function $a(r)=-g(r)$ can be computed by solving the ordinary differential equation
\begin{center}
 $a'(r)\lambda(a(r))=\lambda(r)$
\end{center}
which gives
\begin{equation}
\label{eqDiff}
 \int_{0}^{a(r)}\lambda(s)ds=2-\int_{0}^{r}\lambda(s)ds.
\end{equation}
We define $R(r)=\int_{0}^{r}\lambda(s)ds$, since $r\mapsto R(r)$ is increasing, its inverse $R^{-1}(w)$ is well defined for $w\in[0,1)$.
Thus, we see that $a(r)$ has the form
\begin{equation}
 a(r)=R^{-1}(2-R(r)).
\end{equation}

\end{example}

\subsubsection{Reducing the dimension under radial symmetry}
\label{reducedPB}

In the case where the marginal $\rho(x)=\rho(\vert x\vert)$ is radially symmetric, the multi-marginal problem with Coulomb cost
\begin{equation}\label{dftquant} 
\inf_{\gamma \in \Pi(\rho, \cdots, \rho)} \int_{\R^{dN}} c(x_1, \cdots, x_N) d \gamma(x_1, \cdots, x_N)
\end{equation}
with $c$ the Coulomb cost given by (\ref{eq7}) involves plans on $\R^{dN}$ which is very costly to discretize. Fortunately, thanks to the symmetries of the problem, it can actually be solved by considering a multi-marginal problem only on $\R_+^N$. Let us indeed define for every $(r_1, \cdots, r_N)\in (0, +\infty)^N$:
\begin{equation}\label{defctil}
\tilc (r_1, \cdots, r_N):=\inf \{c(x_1, \cdots, x_N)   \; : \;  \vert x_1 \vert=r_1, \cdots,  \vert x_N \vert=r_N\}.
\end{equation}
Defining $\lambda$ by (\ref{deflambda}) (or  equivalently (\ref{deflambda2})) and defining $\Pi(\lambda, \cdots, \lambda)$ as the set of probability measures on $\R_+^N$ having each marginal equal to $\lambda$, consider 
\begin{equation}\label{dftquantrad} 
\inf_{\tilg\in \Pi(\lambda, \cdots, \lambda)} \int_{\R_+^{N}} \tilc(r_1, \cdots, r_N) d \tilg(r_1, \cdots, r_N).
\end{equation}
We claim that $\inf (\ref{dftquant})=\inf (\ref{dftquantrad})$. The inequality $\inf (\ref{dftquant})\ge\inf (\ref{dftquantrad})$ is easy: take $\gamma \in \Pi(\rho, \cdots, \rho)$ and define its \emph{radial component}  $\tilg$ by
\begin{equation}\label{radcomp}
 \int_{\R_+^{N}} F(r_1, \cdots, r_N) d \tilg(r_1, \cdots, r_N):=\int_{\R^{dN}} F(\vert x_1\vert, \cdots, \vert x_N\vert)d \gamma(x_1, \cdots, x_N), \; \forall F\in C_c(\R_+^{N}),
\end{equation}
it is obvious that $\tilg\in \Pi(\lambda, \cdots, \lambda)$ and since $c(x_1, \cdots, x_N)\ge \tilc(\vert x_1\vert, \cdots, \vert x_N\vert)$, the inequality $\inf (\ref{dftquant})\ge\inf (\ref{dftquantrad})$ easily follows. To show the converse inequality, we  use duality. Indeed, by standard convex duality, we have
\begin{equation}\label{dualdft}
\inf (\ref{dftquant})=\sup_{u} \Big\{ N \int_{\R^d} u(x) \rho(x) dx \; : \; \sum_{i=1}^N u(x_i)  \le  c(x_1, \cdots, x_N) \Big\}
\end{equation}
and similarly
\begin{equation}\label{dualdftrad}
\inf (\ref{dftquantrad})=\sup_{v } \Big\{ N \int_{\R_+} v(r) \lambda(r) dr \; : \; \sum_{i=1}^N v(r_i)  \le  \tilc(r_1, \cdots, r_N) \Big\}.
\end{equation}
Now since $\rho$ is radially symmetric and the constraint of (\ref{dualdft}) is invariant by changing $u$ by $u\circ R$ with $R$ a rotation (see (\ref{eq7})) , there is no loss of generality in restricting the maximization in (\ref{dualdft}) to potentials of the form $u(x_i)=w(r_i)$, but then the constraint  of (\ref{dualdft})  implies that $w$ satisfies the constraint of (\ref{dualdftrad}). Then we have $\inf (\ref{dftquant})=   \sup (\ref{dualdft}) \le \sup (\ref{dualdftrad})=\inf (\ref{dftquantrad})$. Note then that $\gamma\in  \Pi(\rho, \cdots, \rho)$ solves (\ref{dftquant}) if and only if its radial component $\tilg$ solves (\ref{dftquant}) and $c(x_1, \cdots, x_N)=\tilc(\vert x_1\vert, \cdots, \vert x_N\vert)$ $\gamma$-a.e. Therefore (\ref{dftquant}) gives the optimal radial component, whereas the extra condition $c(x_1, \cdots, x_N)=\tilc(\vert x_1\vert, \cdots, \vert x_N\vert)$ $\gamma$-a.e.  gives an information on the angular distribution of $\gamma$.







\section{Iterative Bregman Projections}
\label{AlProj}

Numerics for multi-marginal problems have so far not been extensively developed. Discretizing the multi-marginal problem leads to the linear program (\ref{eq21}) where
the number of constraints grows exponentially in $N$, the number of marginals.
In this section, we present a numerical method which is not based on linear programming techniques, but on an entropic regularization and the so-called alternate projection method. It has  recently been  applied to various optimal transport problems in \cite{CuturiSink} and \cite{iterative}.

The initial idea goes back to  von Neumann \cite{Ne1}, \cite{Ne2} who proved that the sequence obtained by projecting orthogonally iteratively onto two affine subspaces converges to the projection of the initial point onto the intersection  of these affine subspaces. Since the seminal work of Bregman \cite{bregman}, it is by now well-known that one can extend this idea not only to several affine subspaces (the extension to convex sets is due to Dyskstra but we won't use it in the sequel) but also by replacing the euclidean distance by a general Bregman divergence associated to some suitable strictly and differentiable convex function $f$ (possibly with a domain) where we recall that the Bregman divergence associated with $f$ is given by
 \begin{equation}
  D_f(x,y)=f(x)-f(y)-\langle \nabla f(y),x-y\rangle.
 \end{equation}
In what follows, we shall only consider the Bregman divergence (also known as the Kullback-Leibler distance) associated to the Boltzmann/Shannon entropy $f(x):=\sum_i x_i (\log x_i-1)$ for non-negative $x_i$.  This  Bregman divergence (restricted to probabilities i.e. imposing the normalization $\sum_i x_i =1$) is the Kullback-Leibler  distance or relative entropy:
\[D_f(x,y)=\sum_i x_i \log\Big(\frac{ x_i}{y_i} \Big).\]

Bregman distances are used in many other applications most notably image processing, see \cite{osher} for instance. 

\subsection{The Discrete Problem and its Entropic Regularization}
In this section we introduce the discrete problem solved  using the iterative Bregman projections \cite{bregman}.
From now on, we consider the problem (\ref{kant}) 
\begin{equation}
\label{eq13}
 \min_{\gamma_{N}\in\mathcal{C}}\int_{(\mathbb{R}^{d})^N}c(x_{1},\cdots,x_{N})\gamma_{N}(x_{1},\cdots,x_{N})dx_{1}\cdots dx_{N},
\end{equation}
where $N$ is the number of marginals (or electrons), $c(x_{1},...,x_{N})$ is the Coulomb cost, $\gamma_{N}$ the transport plan,  is the probability 
distribution over   $(\mathbb{R}^d)^N$ and 
$\mathcal{C}:=\bigcap_{i=1}^{N}\mathcal{C}_{i}$ with $\mathcal{C}_{i}:=\{ \gamma_N \in Prob\{( \mathbb{R}^d)^N\} \lvert \, \pi_{i}\gamma_N =\rho \}$ (we remind the reader that electrons are indistinguishable so the $N$ marginals coincide with $\rho$).

In order to discretize (\ref{eq13}), we use a discretisation with $M_d$ points 
of the support of the $k$th  electron density as $\{x_{j_k}\}_{j_k=1,\cdots,M_d}$. If the densities  
 $\rho$ are  approximated  by $\sum_{j_k} \rho_{j_k} \delta_{x_{j_k}}$, we get  
\begin{equation}
\label{eq21}
 \min_{\gamma\in\mathcal{C}}\sum_{j_1,\cdots j_{N}} c_{j_1,\cdots,j_N}\gamma_{j_1,\cdots,j_N},
\end{equation}
where    $c_{j_1,\cdots,j_N} = c(x_{j_{1}},\cdots,x_{j_{N}})$  and   the transport plan support for each coordinate is restricted to the points  
$\{x_{j_k}\}_k=1,\cdots,M_d$ thus becoming a $(M_d)^N$ matrix again denoted $\gamma$ with elements 
 $\gamma_{j_1,\cdots,j_N}$. The marginal constraints  $\mathcal{C}_{i}$ becomes 
  \begin{equation}
\label{constraint}
 \mathcal{C}_{i}:=\{ \gamma \in\mathbb{R}_+^{  (M_d)^N }\lvert\quad\sum_{j_{1},...,j_{i-1},j_{i+1},...,j_{N}}\gamma_{j_{1},...,j_{N}}=\rho_{j_i} ,  \,  \forall j_i = 1,\cdots,M_d  \}.
\end{equation}
Recall that the electrons are indistinguishable, meaning that they have same densities~: $ \rho_{j_k}  = \rho_{j_{k'}}, \, \forall j , \, \forall k\ne k' $. 

The discrete optimal transport problem (\ref{eq21}) is  a linear program problem and is dual to  the discretization of (\ref{eq8})
\begin{equation}
 \label{DualDiscrete}
 \begin{split}
  \max_{u_{j}}&\sum_{j=1}^M Nu_{j}\rho_{j}\\
  s.t. & \sum_{i=1}^{N}u_{j_i} \leq c_{j_{1}\cdots j_{N}}\quad\forall j_{i}=1,\cdots,M_d
 \end{split}
 \end{equation}
where $u_{j}=u_{j_{i}}=u(x_{j_i})$.
Thus the primal (\ref{eq21}) has $(M_d)^N$ unknown and $M_d\times N$ linear constraints  and the dual  (\ref{DualDiscrete})  only 
$M_d$ unknown but still $(M_d)^N$ constraints. 
They are computationally not solvable with standard linear programming methods even for small cases in the multi-marginal case. \\

A different approach consists in computing the problem (\ref{eq21}) regularized by the entropy of the joint coupling. This regularization dates to E. Schr\"{o}dinger \cite{Schrodinger31} and it has been recently introduced in machine 
learning \cite{CuturiSink} and economics \cite{GalichonEntr} (we refer the reader to \cite{iterative} for an overview of the entropic regularization and the iterative Bregman projections in OT).
Thus, we consider the following discrete regularized problem
\begin{equation}
 \min_{\gamma\in\mathcal{C}}\sum_{j_1,\cdots j_{N}}c_{j_1,\cdots,j_N}\gamma_{j_1,\cdots,j_N}+\epsilon E(\gamma),
\end{equation}
where $E(\gamma)$ is defined as follows
\begin{equation}
 E(\gamma)=\begin{cases} \sum_{j_1,\cdots j_{N}}\gamma_{j_1,\cdots,j_N}\log(\gamma_{j_1,\cdots,j_N})\mbox{ if }  \gamma\geq 0 \\ +\infty \mbox{ otherwise}.\end{cases}
\end{equation}
After elementary computations, we can re-write the problem as
\begin{equation}
 \label{eq22}
 \min_{\gamma\in \mathcal{C}}KL(\gamma\lvert\bar{\gamma})
\end{equation}
where $KL(\gamma\lvert\bar{\gamma})=\sum_{i_{1},...,i_{N}}\gamma_{i_{1},...,i_{N}}(\log(\dfrac{\gamma_{i_{1},...,i_{N}}}{\bar{\gamma}_{i_{1},...,i{N}}}))$ is the
Kullback-Leibler distance and 
\begin{equation} 
\label{barg} 
\bar{\gamma}_{i_{1},...,i_{N}}=e^{-\dfrac{c_{j_1,\cdots,j_N}}{\epsilon}}.
\end{equation} 

As  explained in section \ref{OT}, when 
the transport plan $\gamma$ is  concentrated  on the graph of a transport map which solves the Monge problem,  after discretisation of the densities, this property is lost along but we still expect the $\gamma$ matrix to be  sparse.
The entropic regularization will  spread the  support and this helps to stabilize the computation:
it defines a strongly convex program with a unique solution $\gamma^{\epsilon}$ which can be obtained through elementary operations (we detail this in section \ref{IPFP}
for both the continuous and discrete framework). The regularized solutions $\gamma^{\epsilon}$ then converge to $\gamma^{\star}$, the solution of (\ref{eq21}) with minimal entropy, as $\epsilon\rightarrow 0$
(see \cite{CominettiAsympt} for a detailed asymptotic analysis and the proof of exponential convergence). Let us now apply the iterative Bregman projections to find the minimizer of (\ref{eq22}).

\subsection{Alternate Projections}

The main idea of the iterative Bregman projections (we call it \textit{Bregman} as the Kullback-Leibler distance is also  called Bregman distance, see \cite{bregman}) is to construct a sequence $\gamma^n$ (which converges to the minimizer of (\ref{eq22})) by alternately projecting on each set $\mathcal{C}_i$
with respect to the Kullback-Leibler distance.
Thus, the iterative KL (or Bregman) projections can be written
\begin{equation}
 \label{eq15}
 \begin{cases}
  \gamma^{0}&=\bar{\gamma} \\
  \gamma^{n}&=P_{\mathcal{C}_{n}}^{KL}(\gamma^{n-1})\quad \forall n>0
 \end{cases}
\end{equation}
where we have extended the indexing of the set by $N-$periodicity such that $\mathcal{C}_{n+N}=\mathcal{C}_{n}\quad\forall n\in\mathbb{N}$ and $P_{\mathcal{C}_{n}}^{KL}$ denotes the KL projection on $\mathcal{C}_n$. \\

The convergence of $\gamma^{n}$ to the  unique solution of (\ref{eq22})  is well known, it actually holds for large classes of Bregman distances and in particular the Kullback-Leibler divergence  as was proved by Bauschke and Lewis \cite{bauschke-lewis}
\begin{center}
 $\gamma^{n}\rightarrow P_{\mathcal{C}}^{KL}(\bar{\gamma})$ as $n\rightarrow\infty$.
\end{center}
\begin{remark}
If the convex sets $\mathcal{C}_i$ are not affine sub-sets (that is not our case), $\gamma^{n}$ converges toward a point of the intersection which is not the KL projection of $\bar{\gamma}$ anymore
so that a correction term is needed as provided by Dykstra's algorithm (we refer the reader to \cite{iterative}).
\end{remark}
The KL projection on $\mathcal{C}_{i}\quad i=1,...,N$ can be computed  explicitly as detailed in the following proposition
\begin{proposition}
 For $\bar{\gamma}\in (\mathbb{R}_+)^{{M_d}^N}$ the projection $P_{\mathcal{C}_{i}}^{KL}(\bar{\gamma})$ is given by
 \begin{equation}
 \label{eq18}
  P_{\mathcal{C}_{i}}^{KL}(\bar{\gamma})_{j_{1},...,j_{N}}=\rho_{j_{i}}\dfrac{\bar{\gamma}_{j_{1},...,j_{N}}}{\sum_{k_{1},...,k_{i-1},k_{i+1},...,k_{N}}\bar{\gamma}_{k_{1},...,k_{N}}}\quad\forall j_{i}=1,...,M_d.
 \end{equation}
\end{proposition}
\begin{proof}
Introducing Lagrange multipliers $\lambda_{j_{i}}$ associated to the constraint $\mathcal{C}_{i}$ 
 \begin{equation}
 \label{eq16}
  \sum_{j_{1},...,j_{i-1},j_{i+1},...,j_{N}}\gamma_{j_{1},...,j_{N}}=\rho_{j_{i}}
 \end{equation}
the KL projection is given by the optimality condition~:
\begin{equation}
 \log(\dfrac{\gamma_{j_{1},...,j_{N}}}{\bar{\gamma}_{j_{1},...,j_{N}}})-\lambda_{j_{i}}=0
\end{equation}
so that
\begin{equation}
 \label{eq17}
 \gamma_{j_{1},...,j_{N}}=C_{j_{i}}\bar{\gamma}_{j_{1},...,j_{N}},
\end{equation}
where $C_{j_{i}}=e^{\lambda_{j_{i}}}$.
If we substitute (\ref{eq17}) in (\ref{eq16}), we get
\begin{equation}
 C_{j_{i}}=\rho_{j_{i}}\dfrac{1}{\sum_{k_{1},...,k_{i-1},k_{i+1},...,k_{N}}\bar{\gamma}_{k_{1},...,k_{N}}}
\end{equation}
which gives (\ref{eq18}).
\end{proof}

\subsection{From the Alternate Projections to the Iterative Proportional  Fitting Procedure }
\label{IPFP}

The alternate projection procedure (\ref{eq15}) is  performed on $M_d^N$ matrices. 
Moreover each projection (\ref{eq18}) involves computing partial sum of this matrix.  The total 
operation cost of each Bregman iteration  scales like $O( M_d^{2N-1})$. \\

In order to reduce the cost of the  problem, we use  an equivalent formulation of the Bregman algorithm known as the Iterative Proportional Fitting Procedure (IPFP).
Let us consider the problem (\ref{eq22}) in a continous measure setting and, for simplicity, 2-marginals framework
\begin{equation}
\label{eq19}
 \min_{\{\gamma\lvert \pi_{1}(\gamma)=\rho,\pi_{2}(\gamma)=\rho\}}\int\log(\dfrac{d\gamma}{d\bar{\gamma}})d\gamma, 
\end{equation}
where $\rho$, $\rho$ and $\bar{\gamma}$ are nonnegative measures.
The aim of the IPFP is to find the KL projection of $\bar{\gamma}$ on $\Pi(\rho,\rho)$ 
(see (\ref{barg}) for the definition of $\bar{\gamma}$  which depends on the cost function). 

Under the assumption  that the value of $(\ref{eq19})$ is finite, R\"{u}schendorf and  Thomsen (see \cite{RuschendorfThomsen}) proved that a unique KL-projection $\gamma^{*}$ exists and that it is of the form
\begin{equation}
 \gamma^{*}(x,y)=a(x)b(y)\bar{\gamma}(x,y),\quad a(x)\geq 0,\quad b(y)\geq 0. 
\end{equation}
From now on, we consider (with a sligthly abuse of notation) Borel measures with densities $\gamma$, $\bar{\gamma}$, $\rho$ and $\rho$ w.r.t. the suitable Lebesgue measure.
$a$ and $b$ can be uniquely determined by the marginal condition as follows
\begin{equation}
\begin{array}{ll}
a(x)&=\dfrac{\rho(x)}{\int\bar{\gamma}(x,y)b(y)dy},\\
b(y)&=\dfrac{\rho(y)}{\int\bar{\gamma}(x,y)a(x)dx}.
\end{array}
\end{equation}
Then,  IPFP is defined by the following recursion
\begin{center}
 $b_{0}=1,\quad a_{0}=\rho$,
\end{center}
\begin{equation}
\label{ip} 
 \begin{array}{ll}
  b_{n+1}(y)&=\dfrac{\rho(y)}{\int\bar{\gamma}(x,y)a_{n}(x)dx},\\
  a_{n+1}(x)&=\dfrac{\rho(x)}{\int\bar{\gamma}(x,y)b_{n+1}(y)dy}.
 \end{array}
\end{equation}
Moreover, we can define the sequence of joint densities (and of the corresponding measures)
\begin{equation}
 \gamma^{2n}(x,y):=a^{n}(x)b^{n}(y)\bar{\gamma}(x,y)\quad\gamma^{2n+1}:=a^{n}(x)b^{n+1}(y)\bar{\gamma}(x,y),\quad n\geq 0. 
\end{equation}
R\"{u}schendorf proved (see \cite{Ruschendorf95}) that $\gamma^{n}$ converges to the KL-projection of $\bar{\gamma}$.
We can, now, recast the IPFP in a discrete framework, which reads as 
\begin{alignat}{3}
&\gamma_{ij}=a_{i}b_{j}\bar{\gamma}_{ij}, \;
& b_{j}^{0}=1,\quad a_{i}^{0}=\rho_{i},
\end{alignat}
\begin{equation}
\label{ip2} 
 \begin{array}{ll}
  b_{j}^{n+1}&=\dfrac{\rho_{j}}{\sum_{i}\bar{\gamma}_{ij}a_{i}^{n}},\\
  a_{i}^{n+1}&=\dfrac{\rho_{i}}{\sum_{j}\bar{\gamma}_{ij}b_{j}^{n+1}},
\end{array}
\end{equation}
\begin{equation}
\label{eq20}
 \gamma_{ij}^{2n}=a_{i}^{n}\bar{\gamma}_{ij}b_{j}^{n}\quad\gamma_{ij}^{2n+1}=a_{i}^{n}\bar{\gamma}_{ij}b_{j}^{n+1}.
\end{equation}
By definition of $\gamma_{ij}^{n}$,  notice that
\begin{center}
$\bar{\gamma}_{ij}b_{j}^{n}=\dfrac{\gamma_{ij}^{2n-1}}{a_{i}^{n-1}}$ and $a_{i}^{n}\bar{\gamma}_{ij}=\dfrac{\gamma_{ij}^{2n}}{b_{j}^{n}}$
\end{center}
and if (\ref{eq20}) is re-written as follows
\begin{equation}
 \begin{array}{ll}
  \gamma_{ij}^{2n}&=\rho_{i}\dfrac{\bar{\gamma}_{ij}b_{j}^{n}}{\sum_{k}\bar{\gamma}_{ik}b_{k}^{n}}\\
  \gamma_{ij}^{2n+1}&=\rho_{j}\dfrac{\bar{\gamma}_{ij}a_{i}^{n}}{\sum_{k}\bar{\gamma}_{kj}a_{k}^{n}} 
 \end{array}
\end{equation}
then we obtain
\begin{equation}
 \begin{array}{ll}
  \gamma_{ij}^{2n}&=\rho_{i}\dfrac{\gamma_{ij}^{2n-1}}{\sum_{k}\gamma_{ik}^{2n-1}}\\
  \gamma_{ij}^{2n+1}&=\rho_{j}\dfrac{\gamma_{ij}^{2n}}{\sum_{k}\gamma_{kj}^{2n}}.
 \end{array}
\end{equation}
Thus, we exactly recover the Bregman algorithm described in the previous section, for 2 marginals. \\

The extension to the multi-marginal framework  is straightforward but 
cumbersone to write. 
It leads to a problem set on $N$  $M_d$-dimensional vectors $a_{j,i_{(\cdot)}},\quad \, j = 1,\cdots,N, \quad i_{(\cdot)}= 1,\cdots,M_d $. Each update takes the form 
\begin{equation}
\label{ip3} 
  a_{j,i_{j}}^{n+1}  =\dfrac{\rho_{i_{j}}}{\sum_{i_1,i_2,... i_{j-1},i_{j+1},...,i_N   }\bar{\gamma}_{i_1,...,i_N} \,
  a_{1,i_1}^{n+1}   \,    a_{2,i_2}^{n+1}  ...   a_{j-1,i_{j-1}}^{n+1}   \,   a_{j+1,i_{j+1}}^{n}   ...   a_{N,{i_N}}^{n}   \,         }, 
\end{equation} 

Where each $i_k$ takes values in $\{1,\cdots,M_d\}$.\\

Note that we still need  a  constant $M_d^N$ cost matrix $\bar{\gamma}$.   Thanks to the symmetry and separability properties of the cost function (see (\ref{eq7}) and (\ref{barg})) , it is possible to 
replace it by a $N\,(N-1)/2$ product of  $M_d^2$ matrices. This is already a big improvement from 
the storage point of view.
Further simplifications are under investigations but the brute force 
IPFP operational cost  therefore scales like $O( N \,  M_d^{N+1} )$ which provides a small improvement over the 
Bregman iterates option. 

\subsection{A heuristic refinement mesh strategy}

We will use a  heuristic refinement mesh strategy allowing to obtain  more accuracy 
without increasing the computational cost and memory requirements. 
This idea was introduced  in \cite{oberman}  for the adaptative resolution of the 
pure Linear Programming formulation of the Optimal Transportation problem, i.e without
the entropic regularisation. \\

If the optimal transport plan is supported by a lower dimensional set, we expect  the entropic regularisation to be concentrated on a  
 mollified version of  this set. Its width should decrease with  the entropic parameter $\epsilon$ if the discretisation is 
fine enough.  Working with a fixed $\epsilon$,  the  idea is to apply coarse to fine progressive resolution and work with a sparse matrix $\gamma$.
At each level,  values below a threshold  are filtered out (set to 0), then new positive values 
are interpolated on a finer grid (next level) where $\gamma$ is strictly positive.  \\

To simplify the exposition, we describe the algorithm for  $2-$marginals in $1D$ and take a $\sqrt{M}$ gridpoints discretization of $I=[a,b]\in\mathbb{R}$:
\begin{enumerate}
 \item we start with a cartesian  $M$ gridpoints mesh on  $I\times I$ to approximate transport plan $\gamma^{\epsilon}$, obtained by running the IPFP on a coarse grid. 

\item  we take $m_c(j)=max_{i}\gamma^{\epsilon}_{ij}$ and $m_r(i)=max_{j}\gamma^{\epsilon}_{ij}$ which are the maximum values over the rows and over the columns respectively, 
and we define 
\begin{center}
 $m=\min[\min_j(m_c(j)), \min_i(m_r(i))]$.
\end{center} 
We will refine the grid only inside the level curve $  \gamma^{\epsilon}  = \xi m$ where we expect the 
finer solution is supported,  see figure  \ref{figure:RegionR}.


\item In order to keep approximately the same number of element in the sparse matrix $\gamma$ at each level 
we refine the grid as follows~:  
Let $ \mathcal{T}:=\{ (i,j) \lvert \gamma^{\epsilon}_{ij} \geq \xi m \}$ and 
$M_{\mathcal{T}}:=\sharp\mathcal{T}$ and $r:=M_{\mathcal{T}}/M$, then the size of the grid at the 
next level is $M^{new}=M/r$.  \\

\item We  compute the interpolation $\gamma_{M^{new}}$ of the old transport plan $\gamma_{M}$  on the finer grid.
\item Elements  of $\gamma_{M^{new}}$ below the fixed threshold $\xi m$  are filtered out, i.e are fixed to $0$ and are not used in the IPFP sum computations, see figure  \ref{figure:RegionR}. 
\item Finally, a new IPFP computation is performed and it can be initialised with an interpolation of the data at the previous level ($\bar{\gamma}$ can be easly re-computed on the gridpoints where $\gamma_{M^{new}}$ is strictly positive). 
\end{enumerate}



\begin{figure}[htbp]

\begin{tabular}{@{}c@{\hspace{1mm}}c@{\hspace{1mm}}c@{}}

\centering
\includegraphics[ scale=0.12]{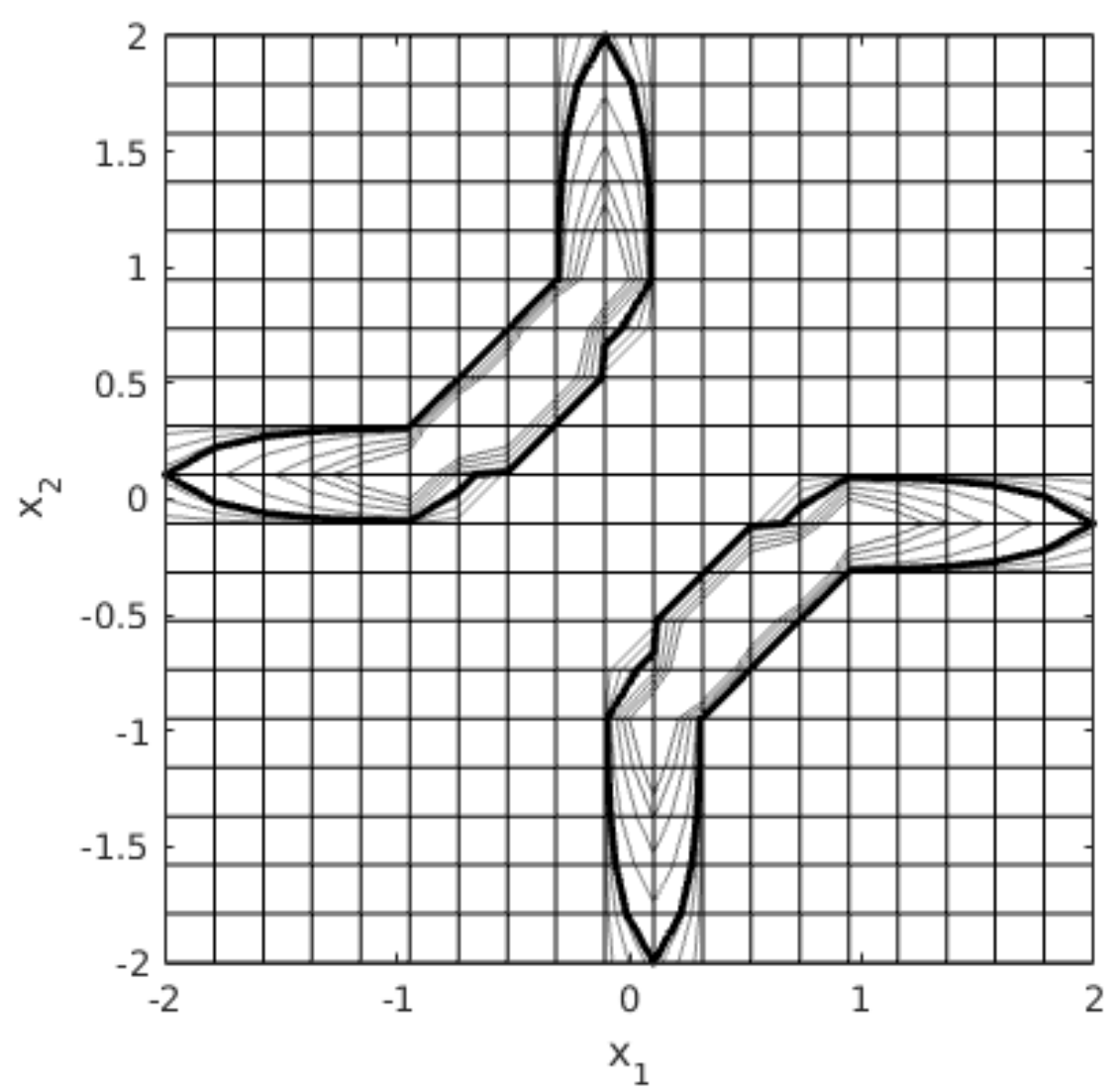}&
\includegraphics[ scale=0.12]{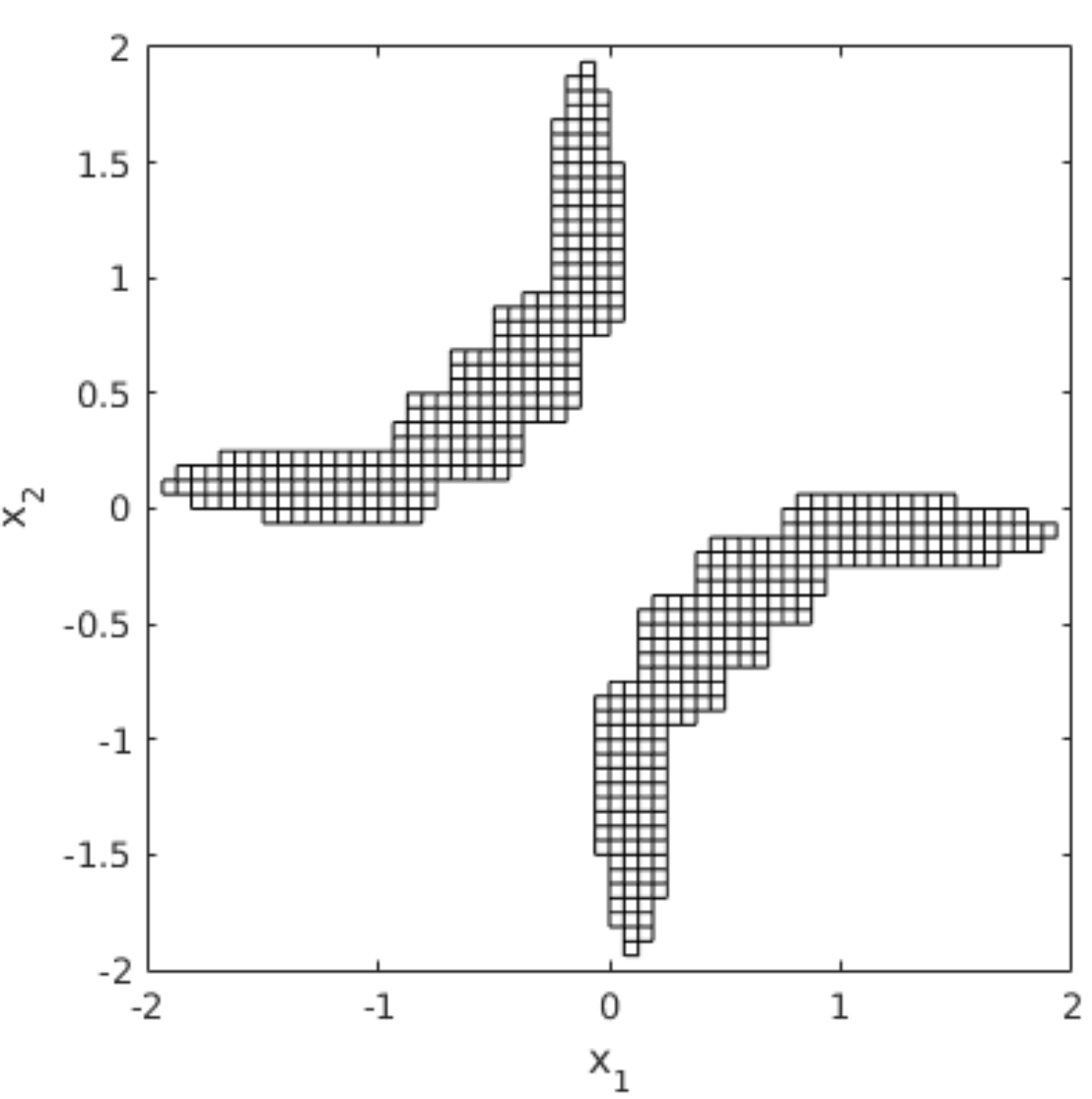}&
\includegraphics[ scale=0.123]{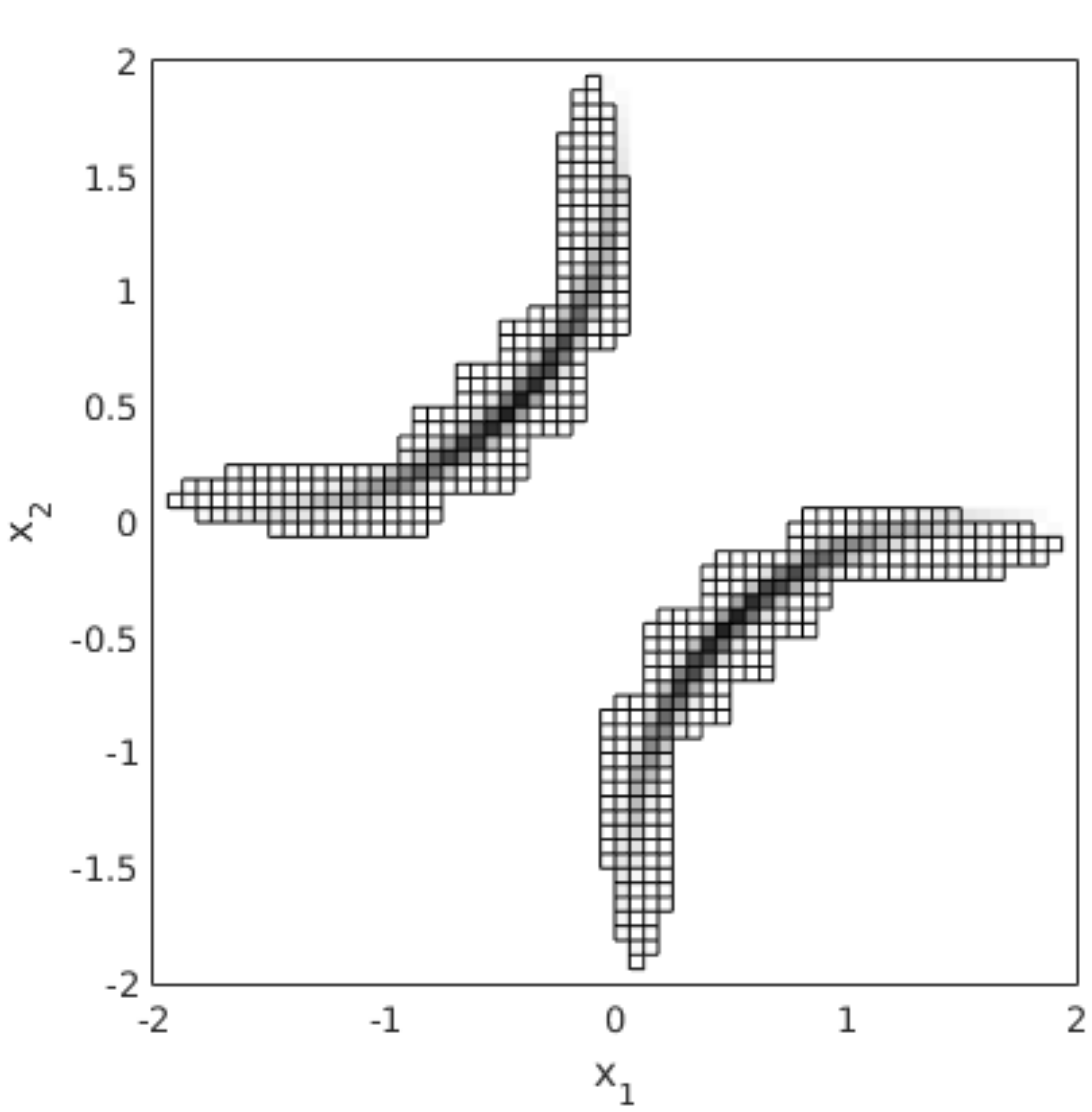}\\

\end{tabular}
\caption{\textit{Left:   $\mathcal{T}$ is the set of grid points inside the level curve  $ \gamma = \xi \, m $ ($\xi=0.9$) (the bold line curve).
Center: The new grid after the refinement. Right: The transport Plan after a new IPFP computation}}
\label{figure:RegionR}
\end{figure}







\section{Numerical Results}
\label{NumRes}


\subsection{$N = 2$ electrons: comparison between numerical and analytical results}
In order to validate the numerical method, we now compare some numerical results for $2$ electrons in dimension $d=1,\cdots,3$ with the analytical results from section \ref{anal}.
Let us first consider a uniform density (as (\ref{eq9}) with $a=2$) in $1D$.
In table \ref{tab:tb1}, we analyze the performance of the numerical method by varying the parameter $\epsilon$.
We  notice that the error becomes smaller by decreasing the regularizing parameter, but the drawback is that the method needs more iterations to converge.
Figure \ref{fig:uniform1D} shows the Kantorovich potential, the co-motion function which can be recovered from the potential by using (\ref{equi}) and the transport plan.
The simulation is performed with a discretization of (\ref{eq9}) with $a=2$, $M=1000$ (gridpoints) and $\epsilon=0.004$.\\

As explained in section \ref{radial}, we can also compute the co-motion for a radially symmetric density.
We have tested the method in $2D$ and $3D$, figure \ref{fig:Uniform2D} and \ref{fig:Uniform3D} respectively,
by using the normalized uniform density on the unit ball.
Moreover, in the radial case we have proved that the OT problem can be reduced to a $1-$dimensional problem by computing $\tilc$ which is trivial for the $2$ electrons
case: let us set the problem in $2D$ in polar coordinates $(r_{1},\theta_{1})$ and $(r_{2},\theta_{2})$, for the first and the second electron respectively (without loss of generality we can set $\theta_{1}=0$), then
it is easy to verify that the minimum is achieved with $\theta_{2}=\pi$.
Figure \ref{fig:Uniform2D} shows the Kantorovich potential (the radial component $v(r)$ as defined in section \ref{reducedPB}), the co-motion and the transport plan for the $2-$dimensional case, the simulation is
performed with $M=1000$ and $\epsilon=0.002$.
In figure \ref{fig:Uniform3D} we present the result for th $3-$dimensional case, the simulation  is
performed with $M=1000$ and $\epsilon=0.002$.

\begin{remark}
 One can notice that, in the case of a uniform density, the transport plan presents a concentration of mass on the boundaries. This is 
 a combined effect of the regularization and of the fact that the density has a compact support.
 \end{remark}

\begin{table}[ht]

 \centering

\begin{tabular}{|c| c| c| c| }

\hline\hline 
$\epsilon$ & Error ($\lVert u^{\epsilon}-u\rVert_{\infty} / \lVert u \rVert_{\infty}$)  & Iteration  \\ [0.5ex]
\hline
0.256 & 0.1529 & 11  \\ \hline
0.128 & 0.0984 & 16 \\ \hline
0.064 & 0.0578 & 25  \\ \hline
0.032 & 0.0313 & 38  \\ \hline
0.016 & 0.0151 & 66 \\ \hline
0.008 & 0.0049 & 114  \\ \hline
0.004 & 0.0045 & 192  \\ \hline

\hline

\end{tabular}
\caption{\textit{Numerical results for uniform density in 1D. $u^{\epsilon}$ is the numerical Kantorovich potential and $u$ is the analytical one.}}
\label{tab:tb1}
\end{table} 

\begin{figure}[ht]

\begin{tabular}{@{}c@{\hspace{1mm}}c@{}}

\centering
\includegraphics[ scale=0.35]{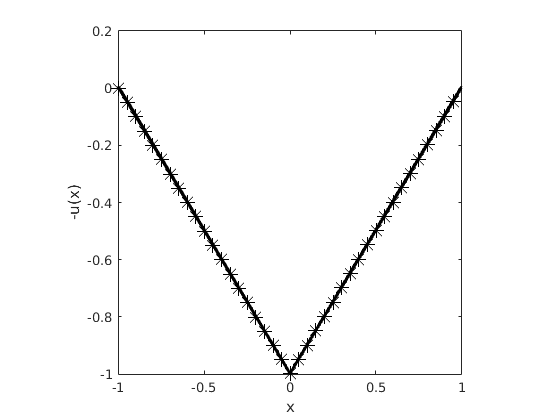}&
\includegraphics[ scale=0.35]{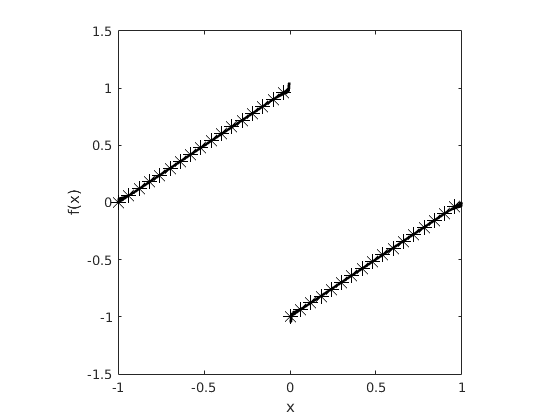} \\
\includegraphics[ scale=0.35]{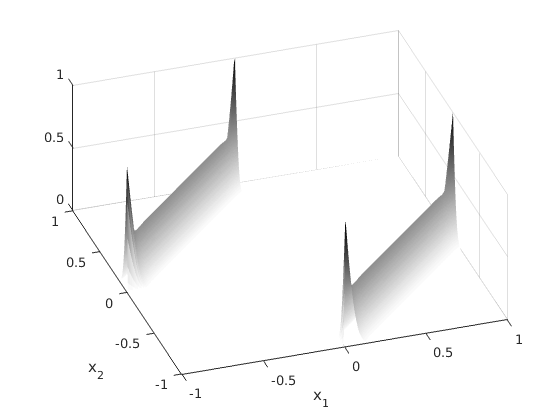} &
\includegraphics[ scale=0.35]{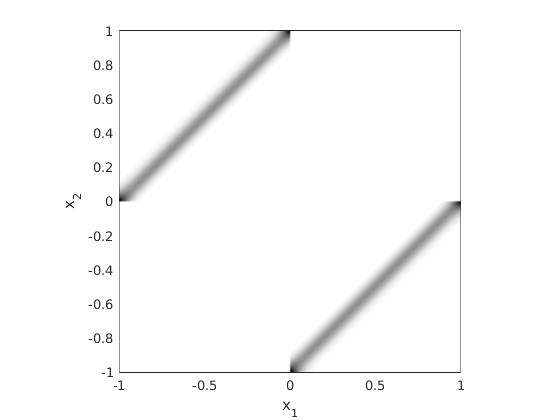}\\
\end{tabular}
\caption{\textit{Top-Left: Kantorovich Potential $u(x)$. Top-Right: Numerical co-motion function (solid line) and analytical co-motion (star-solid line) .  
Bottom-Left: Transport plan $\tilde{\gamma}$. Bottom-Right: Support of $\tilde{\gamma}$.}}
\label{fig:uniform1D}
\end{figure}

\begin{figure}[ht]
\begin{tabular}{@{}c@{\hspace{1mm}}c@{}}
\centering
\includegraphics[ scale=0.35]{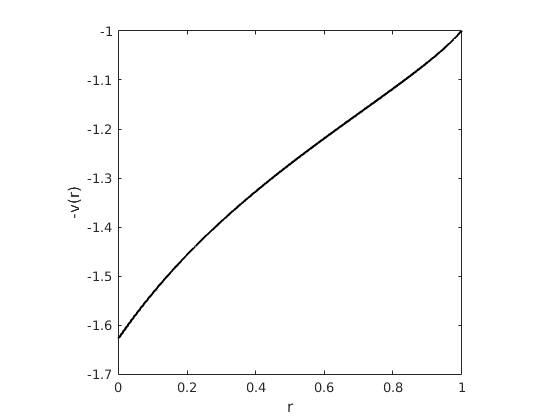} &
\includegraphics[ scale=0.35]{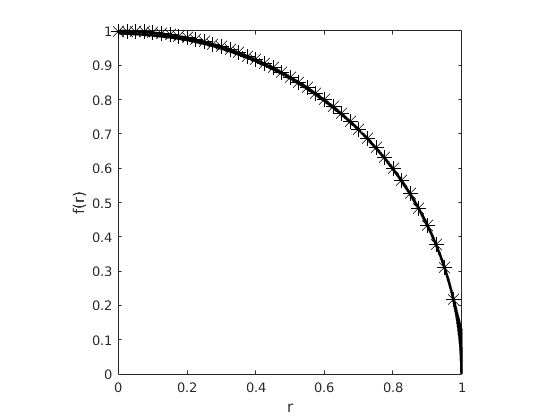}\\
\includegraphics[ scale=0.35]{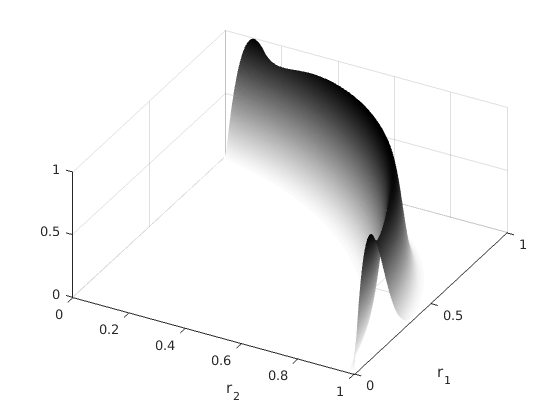} &
\includegraphics[ scale=0.35]{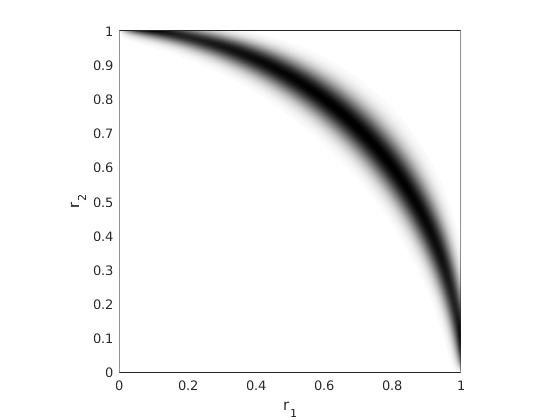}\\

\end{tabular}
\caption{\textit{Top-Left: Kantorovich Potential $v(r)$. Top-Right: Numerical co-motion function (solid line) and analytical co-motion (star-solid line) .  
Bottom-Left: Transport plan $\tilde{\gamma}$. Bottom-Right: Support of $\tilde{\gamma}$.}}
\label{fig:Uniform2D}
\end{figure}

\begin{figure}[ht]
\begin{tabular}{@{}c@{\hspace{1mm}}c@{}}
\centering
\includegraphics[ scale=0.35]{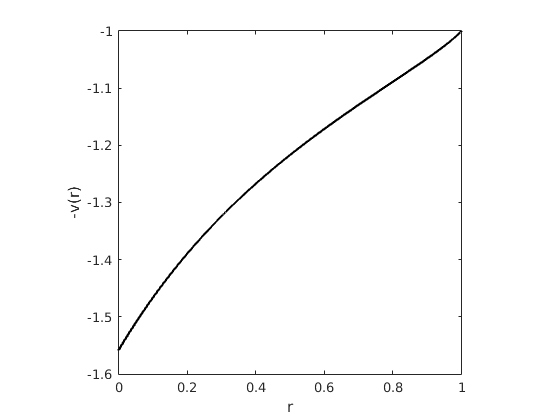} &
\includegraphics[ scale=0.35]{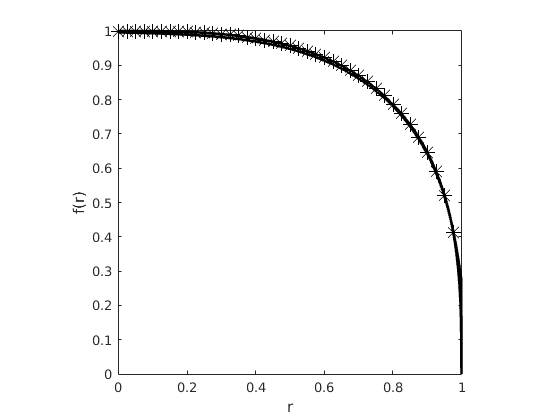}\\
\includegraphics[ scale=0.35]{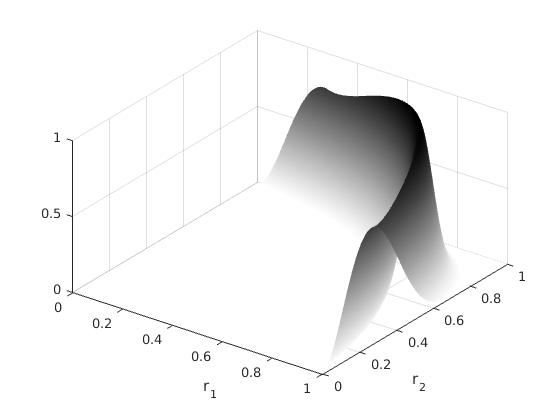} &
\includegraphics[ scale=0.35]{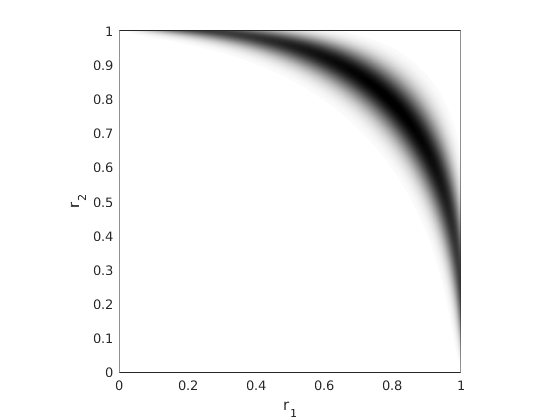}\\

\end{tabular}
\caption{\textit{Top-Left: Kantorovich Potential $v(r)$. Top-Right: Numerical co-motion function (solid line) and analytical co-motion (star-solid line) .  
Bottom-Left: Transport plan $\tilde{\gamma}$. Bottom-Right: Support of $\tilde{\gamma}$.}}
\label{fig:Uniform3D}
\end{figure}

\subsection{$N = 2$ electrons in dimension $d = 3$~:  Helium atom}
Once we have validated the method with some analytical examples, we solve the regularized problem for the Helium atom by using the electron density computed
in \cite{Helium}.
In figure \ref{fig:helium}, we present the electron density, the Kantorovich potential and the transport plan. The simulation is performed
with a discretization of $[0,4]$ with $M=1000$ and $\epsilon=5\, 10^{-3}$.
We can notice the potential correctly fits the asymptotic behaviour from \cite{SeidlGoriSavin}, namely $v(r)\sim \dfrac{N-1}{\lvert r \rvert}$
for $r \rightarrow \infty$, where $N$ is the number of electrons. \begin{figure}[ht]
\begin{tabular}{@{}c@{\hspace{1mm}}c@{}}
\centering
\includegraphics[ scale=0.35]{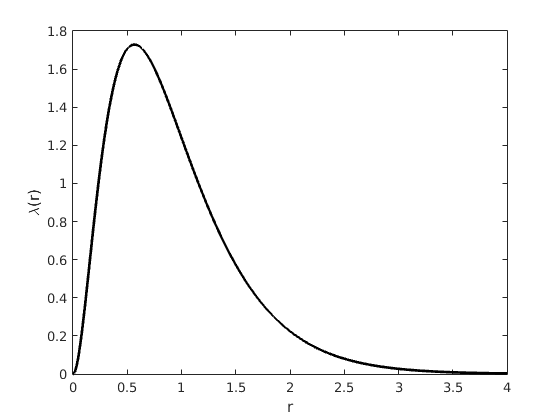} &
\includegraphics[ scale=0.35]{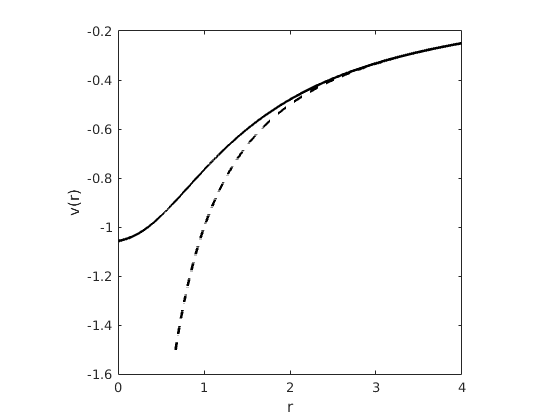}\\
\includegraphics[ scale=0.2]{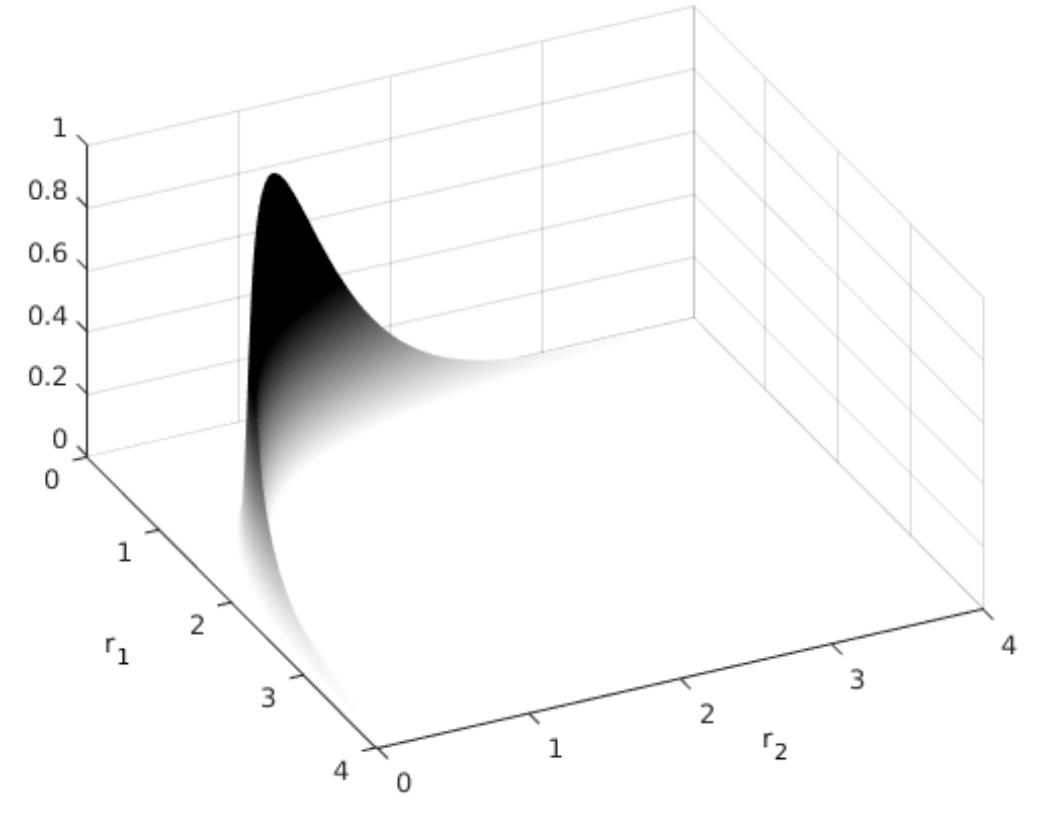} &
\includegraphics[ scale=0.35]{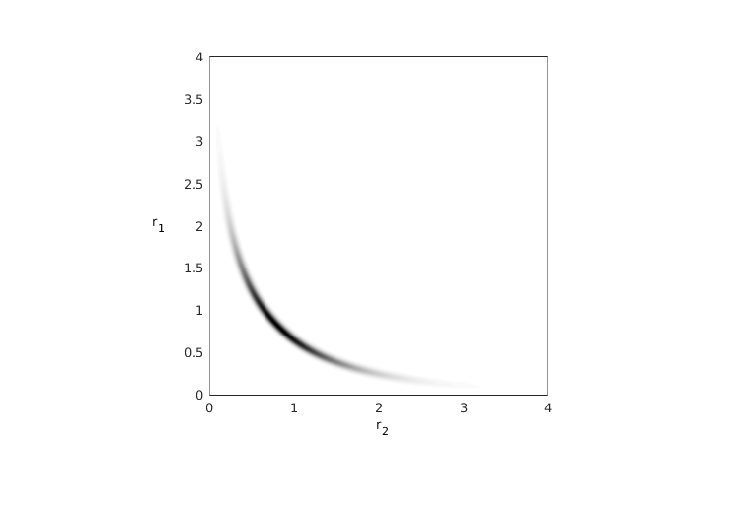}\\

\end{tabular}
\caption{\textit{Top-Left: Helium density $\lambda(r)=4\pi r^{2}\rho(r)$. Top-Right: Kantorovich Potential $v(r)$ (blue) and asymptotic behaviour (red) 
$v(r)\sim\frac{1}{r}\quad r\rightarrow \infty$. Bottom-Left: Transport plan $\tilde{\gamma}$. Bottom-Right: Support of $\tilde{\gamma}$. All quantities are in Hartree atomic units.}}
\label{fig:helium}
\end{figure}

\subsection{$N = 3$ electrons in dimension  $d = 1$ }

We present now some results for the $1-$dimensional multi-marginal problem with $N=3$. They are validated against 
the analytical solutions  given in 
 section \ref{multi1D}. We recall that 
splitting  $\rho$ into three tertiles $\rho_{i}$ with equal mass, we will have  $\rho_{1}\rightarrow\rho_{2}$, $\rho_{2}\rightarrow\rho_{3}$ and
$\rho_{3}\rightarrow\rho_{1}$. \\

In table \ref{tb2}, we present the perfomance of the method for a uniform density on $[0,1]$ by varying $\epsilon$ and, as expected, we see the same behaviour as in the $2$ marginals case.
Figure \ref{fig:3Uniform} shows the Kantorovich potential and the projection of the transport plan onto two marginals (namely $\gamma^{2}=\pi_{12}(\gamma^{\epsilon})$).  The support gives the relative positions of
two electrons.  \\

The simulation is performed on a discretization of $[0,1]$ with a uniform density, $M=1000$ and $\epsilon=0.02$.
If we focus on the support of the projected transport plan we can notice that the numerical solution correctly reproduces the prescribed behavior
The concentration of mass is again due to the compact support of the density, which is not the case of the gaussian as one can see in figure \ref{fig:3Gaussians}.
In figure \ref{fig:3Gaussians} we present the numerical results for $\rho(x)=e^{-x^2}/\sqrt{\pi}$. The simulation is performed on the discretization of $[-2.5,2.5]$
with $M=1000$ and $\epsilon=0.008$.

\begin{table}[ht]

 \centering

\begin{tabular}{|c| c| c| }

\hline\hline 
$\epsilon$ & Error ($\lVert u^{\epsilon}-u\rVert_{\infty} / \lVert u \rVert_{\infty}$)  & Iteration   \\ [0.5ex]
\hline
0.32 & 0.0658 & 121   \\ \hline
0.16 & 0.0373 & 230   \\ \hline
0.08 & 0.0198 & 446   \\ \hline
0.04 & 0.0097 & 878   \\ \hline
0.02 & 0.0040 & 1714   \\ \hline

\hline

\end{tabular}
\caption{\textit{Numerical results for uniform density in 1D and three electrons. $u^{\epsilon}$ is the numerical Kantorovich potential and $u$ is the analytical one.}}
\label{tb2}
\end{table} 

\begin{figure}[ht]
\begin{tabular}{@{}c@{\hspace{3mm}}c@{}}
\centering
\includegraphics[ scale=0.25]{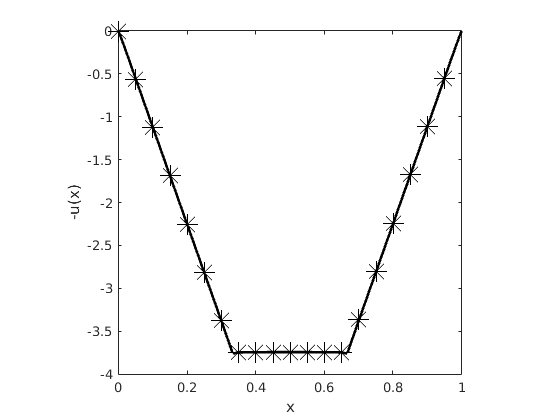}
\includegraphics[ scale=0.25]{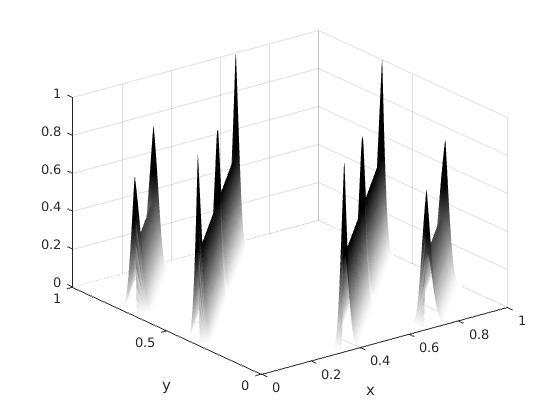} &
\includegraphics[ scale=0.25]{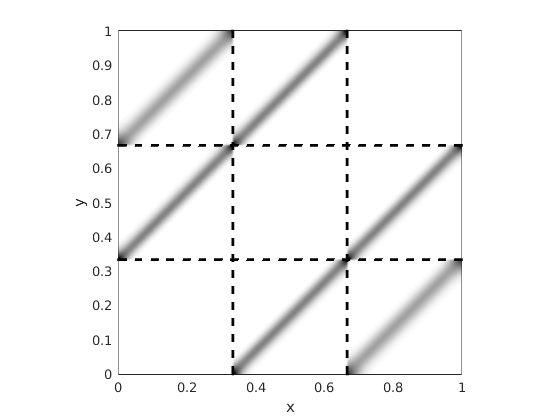}\\
\end{tabular}
\caption{\textit{Left: Numerical Kantorovich potential $u(x)$ (solid line) and analytical potential (star-solid line). Center: Projection of the transport plan $\pi_{12}(\gamma(x,y,z))$. Rigth: Support of $\pi_{12}(\gamma(x,y,z))$
The dot-dashed lines delimit the intervals where 
$\rho_{i}$, with $i=1,\cdots,3$, are defined.}}
\label{fig:3Uniform}
\end{figure}

\begin{figure}[htbp]
\begin{tabular}{@{}c@{\hspace{3mm}}c@{}c@{\hspace{3mm}}c@{}}
\centering
\includegraphics[ scale=0.25]{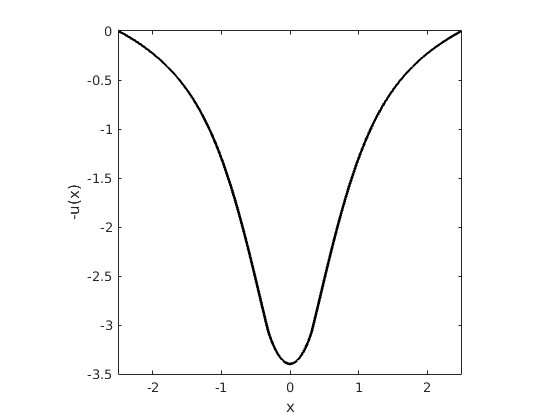}
\includegraphics[ scale=0.25]{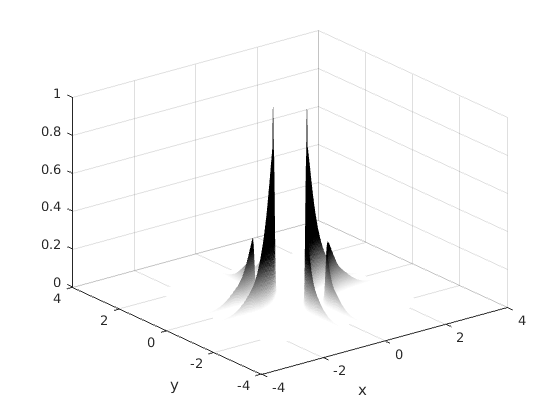} &
\includegraphics[ scale=0.25]{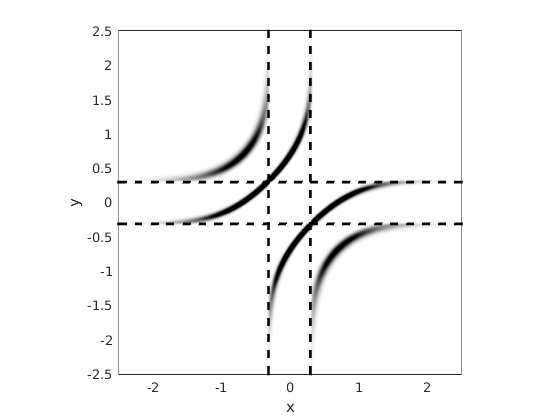}\\
\end{tabular}
\caption{\textit{Left: Kantorovich potential $u(x)$. Center: Projection of the transport plan $\pi_{12}(\gamma(x,y,z))$. Rigth: Support of $\pi_{12}(\gamma(x,y,z))$. The dot-dashed lines delimit the intervals where 
$\rho_{i}$, with $i=1,\cdots,3$, are defined.}}
\label{fig:3Gaussians}
\end{figure}

\subsection{$N = 3$ electrons in dimension  $d = 3$ radial case~:  Litium atom}
We finally perform some simulations for the radial $3-$dimensional case for $N=3$.
As for the $3-$dimensional case with $2$ marginals  we solve the reduced problem: let us consider the spherical coordinates $(r_{i},\theta_{i},\phi_{i})$
with $i=1,\cdots,3$ and we fix $\theta_1=0$ and $\phi_1=\phi_2=0$ (the first electrons defines the z axis and the second one is on the xz plane). 
We then notice that  $\phi_3=0$ as the electrons must be on the same plane of the nucleus to achieve compensation of forces (one can see it
by computing the optimality conditions), so we have to minimize on $\theta_2$ and $\theta_3$ in order to obtain $\tilc$.\\

Figure \ref{fig:litium} shows the electron density of the Litium (computed in \cite{litium}), the Kantorovich Potential (and the asymptotic behavior) and the
projection of the transport plan onto two marginals $\tilde{\gamma}^{2}=\pi_{12}(\tilde{\gamma}^{\epsilon})$.  The support gives  the relative positions of
two electrons.  \\

The simulation is performed on a discretization of $[0,8]$  with $M=300$ and $\epsilon=0.007$.
Our results show (taking into account the regularization effect) a concentrated transport plan for this kind of density and they match analogous result
obtained in \cite{SeidlGoriSavin}. If we focus on the support of the transport plan we can notice that the optimal solution forces the electrons to occupy
three different regions as conjectured in \cite{SeidlGoriSavin}.

\begin{figure}[ht]
\begin{tabular}{@{}c@{\hspace{1mm}}c@{}}
\centering
\includegraphics[ scale=0.35]{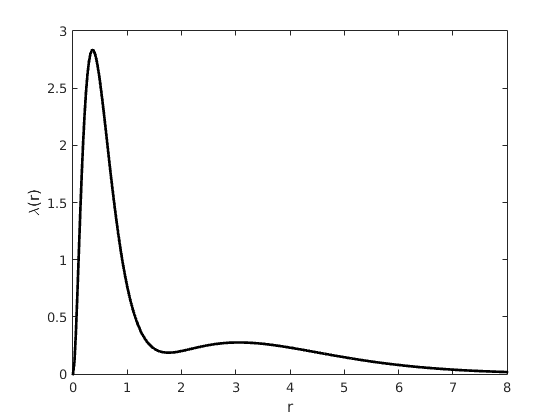} &
\includegraphics[ scale=0.35]{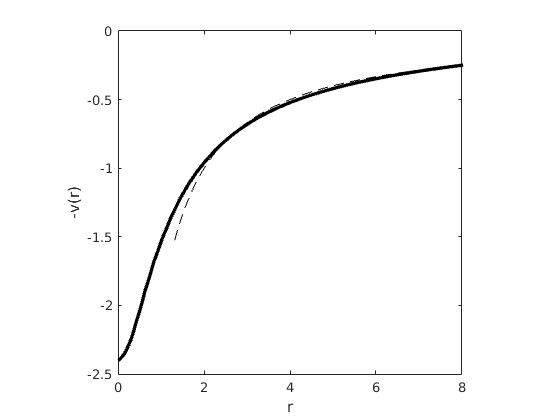}\\
\includegraphics[ scale=0.35]{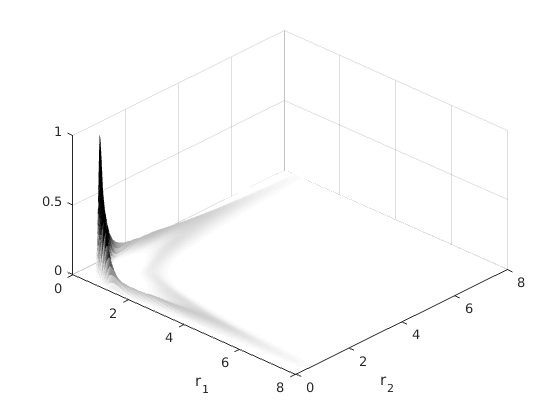} &
\includegraphics[ scale=0.35]{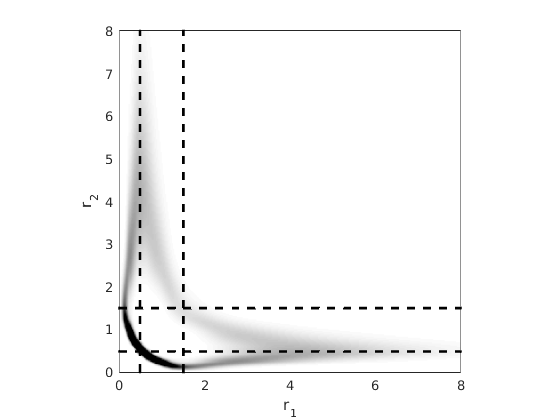}\\

\end{tabular}
\caption{\textit{Top-Left: Litium density $\lambda(r)=4\pi r^{2}\rho(r)$. Top-Right: Kantorovich Potential $v(r)$ (blue) and asymptotic behaviour (red) 
$v(r)\sim\frac{2}{r}\quad r\rightarrow \infty$. Bottom-Left: Projection of the Transport plan $\tilde{\gamma}^{2}=\pi_{12}(\tilde{\gamma}^{\epsilon})$.
Bottom-Right: Support of the projected transport plan $\tilde{\gamma}^{2}$. The dot-dashed lines delimit the three regions that the electrons must occupy, we computed them
numerically following the idea in \cite{SeidlGoriSavin}.All quantities are in Hartree atomic units.}}
\label{fig:litium}
\end{figure}

\section{conclusion}\label{ccl}

We have presented a numerical scheme for solving multi-marginal OT problems arising from DFT. This is  a challenging problems, not only because of the unusual features of the Coulomb cost which is singular and repulsive but also due to the high dimension of the space of plans. 

Using an entropic regularization gives rise to a Kullback-Leibler projection problem onto the intersection of affine subsets given by the marginal constraints.  Because each projection is explicit, one can use Bregman's iterative projection algorithm to approximate the solution. 

The power of such an iterative projection approach was recently emphasized in \cite{CuturiSink, iterative} for the entropic regularization of optimal transport problems, we showed that is also well suited to treat the multi-marginal OT problem with Coulomb cost and leads to the 
same benefits in terms of convexification of the problem and simplicity of implemention. 

The method presented here is just a preliminary step which is simple to implement and therefore easy to use in practice. 
The cost of solving the general DFT problem in dimension 3 for a large number of electrons is still unfeasible and we need to 
use radial symmetry simplification and also a heuristic refinement mesh strategy.

A lot of questions  are left for future research~: can IPFP be used for sharper approximations for DFT? 
Can one justify rigorously and quantitatively the mesh refinement strategy?
How should the regularization parameter $\varepsilon$ be chosen in practice? Does the entropic regularization have a physical interpretation?

\section*{Acknowledgements}

We would like to thank Adam Oberman and Brendan Pass for many helpful and stimulating discussions as well as
Paola Gori-Giorgi for sharing numerical details concerning the Helium and Litium atom.

We gratefully acknowledge the support of the ANR, through the project ISOTACE (ANR-12-MONU-0013) and INRIA through the ``action exploratoire" MOKAPLAN.

\bibliographystyle{spmpsci}
\bibliography{refs}

\begin{thebibliography}{10}
\providecommand{\url}[1]{{#1}}
\providecommand{\urlprefix}{URL }
\expandafter\ifx\csname urlstyle\endcsname\relax
  \providecommand{\doi}[1]{DOI~\discretionary{}{}{}#1}\else
  \providecommand{\doi}{DOI~\discretionary{}{}{}\begingroup
  \urlstyle{rm}\Url}\fi

\bibitem{bauschke-lewis}
Bauschke, H.H., Lewis, A.S.: Dykstra's algorithm with {B}regman projections: a
  convergence proof.
\newblock Optimization \textbf{48}(4), 409--427 (2000)

\bibitem{iterative}
Benamou, J.D., Carlier, G., Cuturi, M., Nenna, L., Peyr{\'e}, G.: Iterative
  bregman projections for regularized transportation problems.
\newblock arXiv preprint arXiv:1412.5154  (2014)

\bibitem{bregman}
Bregman, L.M.: The relaxation method of finding the common point of convex sets
  and its application to the solution of problems in convex programming.
\newblock USSR computational mathematics and mathematical physics
  \textbf{7}(3), 200--217 (1967)

\bibitem{bre}
Brenier, Y.: Polar factorization and monotone rearrangement of vector-valued
  functions.
\newblock Comm. Pure Appl. Math. \textbf{44}(4), 375--417 (1991).
\newblock \doi{10.1002/cpa.3160440402}.
\newblock \urlprefix\url{http://dx.doi.org/10.1002/cpa.3160440402}

\bibitem{litium}
Bunge, C.: The full ci density of the li atom has been computed with a very
  large basis set with 8 s functions and up to k functions.
\newblock private communication

\bibitem{ButDeGo}
Buttazzo, G., De~Pascale, L., Gori-Giorgi, P.: Optimal-transport formulation of
  electronic density-functional theory.
\newblock Phys. Rev. A \textbf{85}, 062,502 (2012).
\newblock \doi{10.1103/PhysRevA.85.062502}.
\newblock \urlprefix\url{http://link.aps.org/doi/10.1103/PhysRevA.85.062502}

\bibitem{carlierekeland}
Carlier, G., Ekeland, I.: Matching for teams.
\newblock Econom. Theory \textbf{42}(2), 397--418 (2010).
\newblock \doi{10.1007/s00199-008-0415-z}.
\newblock \urlprefix\url{http://dx.doi.org/10.1007/s00199-008-0415-z}

\bibitem{CPM}
Colombo, M., De~Pascale, L., Di~Marino, S.: Multimarginal optimal transport
  maps for one-dimensional repulsive costs.
\newblock Canad. J. Math. \textbf{67}, 350--368 (2015)

\bibitem{CominettiAsympt}
Cominetti, R., Martin, J.S.: Asymptotic analysis of the exponential penalty
  trajectory in linear programming.
\newblock Mathematical Programming \textbf{67}(1-3), 169--187 (1994)

\bibitem{CoFrieKl}
Cotar, C., Friesecke, G., Kl\"uppelberg, C.: Density functional theory and
  optimal transportation with {C}oulomb cost.
\newblock Communications on Pure and Applied Mathematics \textbf{66}(4),
  548--599 (2013).
\newblock \doi{10.1002/cpa.21437}.
\newblock \urlprefix\url{http://dx.doi.org/10.1002/cpa.21437}

\bibitem{cotar2013infinite}
Cotar, C., Friesecke, G., Pass, B.: Infinite-body optimal transport with
  coulomb cost.
\newblock Calculus of Variations and Partial Differential Equations pp. 1--26
  (2013)

\bibitem{CuturiSink}
Cuturi, M.: Sinkhorn distances: Lightspeed computation of optimal transport.
\newblock In: Advances in Neural Information Processing Systems (NIPS) 26, pp.
  2292--2300 (2013)

\bibitem{Helium}
Freund, D.E., Huxtable, B.D., Morgan, J.D.: Variational calculations on the
  helium isoelectronic sequence.
\newblock Phys. Rev. A \textbf{29}, 980--982 (1984).
\newblock \doi{10.1103/PhysRevA.29.980}.
\newblock \urlprefix\url{http://link.aps.org/doi/10.1103/PhysRevA.29.980}

\bibitem{friesecke2013n}
Friesecke, G., Mendl, C.B., Pass, B., Cotar, C., Kl{\"u}ppelberg, C.: N-density
  representability and the optimal transport limit of the hohenberg-kohn
  functional.
\newblock The Journal of chemical physics \textbf{139}(16), 164,109 (2013)

\bibitem{GalichonEntr}
Galichon, A., Salani\'e, B.: Matching with trade-offs: Revealed preferences
  over competing characteristics.
\newblock Tech. rep., Preprint SSRN-1487307 (2009)

\bibitem{gansw}
Gangbo, W., {\'S}wi{\c{e}}ch, A.: Optimal maps for the multidimensional
  {M}onge-{K}antorovich problem.
\newblock Comm. Pure Appl. Math. \textbf{51}(1), 23--45 (1998).
\newblock \doi{10.1002/(SICI)1097-0312(199801)51:1<23::AID-CPA2>3.0.CO;2-H}.
\newblock
  \urlprefix\url{http://dx.doi.org/10.1002/(SICI)1097-0312(199801)51:1<23::AID-CPA2>3.0.CO;2-H}

\bibitem{ghou-mau}
Ghoussoub, N., Maurey, B.: Remarks on multi-marginal symmetric
  {M}onge-{K}antorovich problems.
\newblock Discrete Contin. Dyn. Syst. \textbf{34}(4), 1465--1480 (2014)

\bibitem{osher}
Goldstein, T., Bresson, X., Osher, S.: Geometric applications of the split
  bregman method: Segmentation and surface reconstruction.
\newblock Journal of Scientific Computing \textbf{45}(1-3), 272--293 (2010).
\newblock \doi{10.1007/s10915-009-9331-z}.
\newblock \urlprefix\url{http://dx.doi.org/10.1007/s10915-009-9331-z}

\bibitem{HK}
Hohenberg, P., Kohn, W.: Inhomogeneous electron gas.
\newblock Phys. Rev. \textbf{136}, B864--B871 (1964).
\newblock \doi{10.1103/PhysRev.136.B864}.
\newblock \urlprefix\url{http://link.aps.org/doi/10.1103/PhysRev.136.B864}

\bibitem{Kanth42}
Kantorovich, L.: On the transfer of masses (in russian).
\newblock Doklady Akademii Nauk \textbf{37}(2), 227--229 (1942)

\bibitem{KS}
Kohn, W., Sham, L.J.: Self-consistent equations including exchange and
  correlation effects.
\newblock Phys. Rev. \textbf{140}, A1133--A1138 (1965).
\newblock \doi{10.1103/PhysRev.140.A1133}.
\newblock \urlprefix\url{http://link.aps.org/doi/10.1103/PhysRev.140.A1133}

\bibitem{MaletGori}
Malet, F., Gori-Giorgi, P.: Strong correlation in kohn-sham density functional
  theory.
\newblock Phys. Rev. Lett. \textbf{109}, 246,402 (2012).
\newblock \doi{10.1103/PhysRevLett.109.246402}.
\newblock
  \urlprefix\url{http://link.aps.org/doi/10.1103/PhysRevLett.109.246402}

\bibitem{mendl2013kantorovich}
Mendl, C.B., Lin, L.: Kantorovich dual solution for strictly correlated
  electrons in atoms and molecules.
\newblock Physical Review B \textbf{87}(12), 125,106 (2013)

\bibitem{Monge}
Monge, G.: M{\'e}moire sur la th{\'e}orie des d{\'e}blais et des remblais.
\newblock De l'Imprimerie Royale (1781)

\bibitem{Ne2}
von Neumann, J.: Functional Operators. {Vol}. 1: {Measures} and integrals.
  {Vol} 2: {The} geometry of orthogonal spaces.
\newblock 21, 22 (1950--1951)

\bibitem{Ne1}
Neumann, J.V.: On rings of operators. reduction theory.
\newblock Annals of Mathematics \textbf{50}(2), pp. 401--485 (1949).
\newblock \urlprefix\url{http://www.jstor.org/stable/1969463}

\bibitem{oberman}
Oberman, A.: private communication.
\newblock paper in preparation

\bibitem{pass}
Pass, B.: Uniqueness and {Monge} solutions in the multimarginal optimal
  transportation problem.
\newblock SIAM Journal on Mathematical Analysis \textbf{43}(6), 2758--2775
  (2011).
\newblock \doi{10.1137/100804917}.
\newblock \urlprefix\url{http://link.aip.org/link/?SJM/43/2758/1}

\bibitem{Pass-match}
Pass, B.: Multi-marginal optimal transport and multi-agent matching problems:
  uniqueness and structure of solutions.
\newblock Discrete Contin. Dyn. Syst. \textbf{34}(4), 1623--1639 (2014).
\newblock \doi{10.3934/dcds.2014.34.1623}.
\newblock \urlprefix\url{http://dx.doi.org/10.3934/dcds.2014.34.1623}

\bibitem{Ruschendorf95}
Ruschendorf, L.: Convergence of the iterative proportional fitting procedure.
\newblock The Annals of Statistics \textbf{23}(4), 1160--1174 (1995)

\bibitem{RuschendorfThomsen}
Ruschendorf, L., Thomsen, W.: Closedness of sum spaces and the generalized
  {Schrodinger} problem.
\newblock Theory of Probability and its Applications \textbf{42}(3), 483--494
  (1998)

\bibitem{Schrodinger31}
Schrodinger, E.: Uber die umkehrung der naturgesetze.
\newblock Sitzungsberichte Preuss. Akad. Wiss. Berlin. Phys. Math.
  \textbf{144}, 144--153 (1931)

\bibitem{SeidlGoriSavin}
Seidl, M., Gori-Giorgi, P., Savin, A.: Strictly correlated electrons in
  density-functional theory: A general formulation with applications to
  spherical densities.
\newblock Phys. Rev. A \textbf{75}, 042,511 (2007).
\newblock \doi{10.1103/PhysRevA.75.042511}.
\newblock \urlprefix\url{http://link.aps.org/doi/10.1103/PhysRevA.75.042511}

\bibitem{Villani03}
Villani, C.: Topics in Optimal Transportation.
\newblock Graduate Studies in Mathematics Series. American Mathematical Society
  (2003).
\newblock \urlprefix\url{http://books.google.fr/books?id=GqRXYFxe0l0C}

\bibitem{Villani09}
Villani, C.: Optimal transport: old and new, vol. 338.
\newblock Springer (2009)

\end{thebibliography}

\end{document}